\documentclass[12pt, a4paper]{article}
\usepackage{amssymb,amsmath,amsfonts,amsthm}
\usepackage{setspace}
\usepackage[DIV15]{typearea}
\usepackage{tikz}

\usepackage{hyperref}
\providecommand{\url}[1]{}

\DeclareMathOperator{\Var}{V}
\DeclareMathOperator{\Cov}{Cov}
\DeclareMathOperator{\sign}{sign}
\DeclareMathOperator{\BV}{BV}
\numberwithin{equation}{section}
\newcommand{\tq}{\,:\,}

\newcommand{\Monm}{\mathcal{F}}
\newcommand{\tMonm}{\widetilde{\mathcal{F}}}
\DeclareMathOperator{\Mon}{Mon}
\DeclareMathOperator{\tMon}{\widetilde{Mon}}

\newcommand{\ra}{\rightarrow}
\newcommand{\E}{\mathbb{E}}

\newcommand{\p}{ { \mathbb P} }

\newcommand{\R}{ { \mathbb R} }
\newcommand{\N}{ { \mathbb N} }

\newcommand{\F}{\mathcal{F}}
\newcommand{\M}{\mathcal{M}}

\usepackage{dsfont}
\newcommand{\I}{\mathds{1}}

\theoremstyle{plain}
\newtheorem{thm}{Theorem}[section]
\newtheorem{lma}[thm]{Lemma}

\newtheorem{prop}[thm]{Proposition}

\newtheorem{defn}[thm]{Definition}

\theoremstyle{definition}

\newtheorem{rmkm}[thm]{Remark}

\begin{document}

\title{\large \bf Some almost sure results for unbounded functions of
intermittent maps and their associated Markov chains}

\author{\Large J. Dedecker\footnote{Universit\'{e} Paris 6-Pierre et Marie Curie, Laboratoire de Statistique
Th\'{e}orique et Appliqu\'{e}e.}, S. Gou\"{e}zel\footnote{Universit\'{e} Rennes
1, IRMAR and CNRS UMR 6625.} and F.
Merlev\`{e}de\footnote{Universit\'{e} Paris Est-Marne la Vall\'{e}e, LAMA
and CNRS UMR 8050.}}
\date{June 22, 2009}

\maketitle

\begin{abstract}
\vskip 0.15cm We consider a large class of piecewise expanding
maps $T$ of $[0,1]$ with a neutral fixed point, and their
associated Markov chain $Y_i$ whose transition kernel is the
Perron-Frobenius operator of $T$ with respect to the absolutely
continuous invariant probability measure. We give a large class
of unbounded functions $f$ for which the partial sums of
$f\circ T^i$ satisfy both a central limit theorem and a bounded
law of the iterated logarithm. For the same class, we prove
that the partial sums of $f(Y_i)$ satisfy a strong invariance
principle. When the class is larger, so that the partial sums
of $f\circ T^i$ may belong to the domain of normal attraction
of a stable law of index $p\in (1, 2)$, we show that the almost
sure rates of convergence in the strong law of large numbers
are the same as in the corresponding i.i.d.~case.

\end{abstract}

\section{Introduction and main results}

\subsection{Introduction}

The Pomeau-Manneville map is an explicit map of the interval
$[0,1]$, with a neutral fixed point at $0$ and a prescribed
behavior there. The statistical properties of this map are very
well known when one considers H\"{o}lder continuous observables,
but much less is known for more complicated observables.

Our goal in this paper is twofold. First, we obtain optimal
bounds for the behavior of functions of bounded variation with
respect to iteration of the Pomeau-Manneville map. Second, we
use these bounds to get a bounded law of the iterated logarithm
for a very large class of observables, that previous techniques
were unable to handle.

Since we use bounded variation functions, our arguments do not
rely on any kind of Markov partition for the map $T$.
Therefore, it turns out that our results hold for a larger
class of maps, that we now describe.

\begin{defn}
A map $T:[0,1] \to [0,1]$ is a generalized Pomeau-Manneville
map (or GPM map) of parameter $\gamma \in (0,1)$ if there exist
$0=y_0<y_1<\dots<y_d=1$ such that, writing $I_k=(y_k,y_{k+1})$,
\begin{enumerate}
\item The restriction of $T$ to $I_k$ admits a $C^1$ extension
$T_{(k)}$ to $\overline{I_k}$.
\item For $k\geq 1$, $T_{(k)}$ is $C^2$ on $\overline{I_k}$, and $|T_{(k)}'|>1$.
\item $T_{(0)}$ is $C^2$ on $(0, y_1]$, with $T_{(0)}'(x)>1$ for $x\in
(0,y_1]$, $T_{(0)}'(0)=1$ and $T_{(0)}''(x) \sim c
x^{\gamma-1}$ when $x\to 0$, for some $c>0$.
\item $T$ is topologically transitive.
\end{enumerate}
\end{defn}
The third condition ensures that $0$ is a neutral fixed point
of $T$, with $T(x)=x+c' x^{1+\gamma} (1+o(1))$ when $x\to 0$.
The fourth condition is necessary to avoid situations where
there are several absolutely continuous invariant measures, or
where the neutral fixed point does not belong to the support of
the absolutely continuous invariant measure.

\begin{figure}[htb]
\centering
  \begin{tikzpicture}[scale=1.5]
  \draw[very thick] (0,0) rectangle (5,5);
  \draw[gray, very thin]
      (1.2,0) -- +(0, 5)
      (2.6,0) -- +(0, 5)
      (4.2,0) -- +(0, 5)
      (0,0)   -- (5,5);
  \draw[thick]
      (0,0) .. controls +(45:1) and +(-100:1) .. (1.2, 3.2)
      (1.2, 4) -- (2.6, 1)
      (2.6, 1.3) .. controls +(55:1) and +(-120:2) .. (4.2, 5)
      (4.2,4.6) .. controls +(-85:3) and +(105:1) .. (5, 0);
  \foreach \x/\ytext in {0/$y_0=0$, 1.2/$y_1$, 2.6/$y_2$, 4.2/$y_3$, 5/$y_4=1$}
       \node[below, text height=1.5ex] at (\x, 0) {\ytext};
  \end{tikzpicture}
\caption{The graph of a GPM map, with $d=4$}
\end{figure}
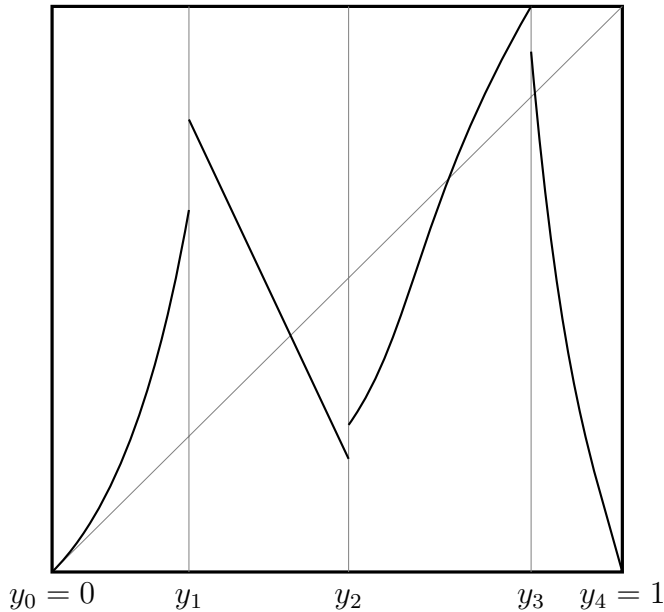

A well known GPM map is the original Pomeau-Manneville map
(1980). The Liverani-Saussol-Vaienti (1999) map
  \[
   T_\gamma(x)=
  \begin{cases}
  x(1+ 2^\gamma x^\gamma) \quad  \text{ if $x \in [0, 1/2]$}\\
  2x-1 \quad \quad \quad \ \  \text{if $x \in (1/2, 1]$}
  \end{cases}
  \]
is also a much studied GPM map of parameter $\gamma$. Both of
them have a Markov partition, but this is not the case in
general for GPM maps as defined above.

Theorem 1 in Zweim\"{u}ller (1998)\footnote{This theorem does not
apply directly to our maps since they do not satisfy its
assumption (A). However, this assumption is only used to show
that the jump transformation $\tilde T$ satisfies (AFU), and
this follows in our setting from the distortion estimates of
Lemma 5 in Young (1999).} shows that a GPM map $T$ admits a
unique absolutely continuous invariant probability measure
$\nu$, with density $h_\nu$. Moreover, it is ergodic, has full
support, and $h_\nu(x)/x^{-\gamma}$ is bounded from above and
below.

From the ergodic theorem, we know that $S_n(f)=n^{-1}
\sum_{i=0}^{n-1} (f\circ T^i - \nu(f))$ converges almost
everywhere to 0 when the function $f:[0,1]\to \R$ is
integrable. If $f$ is H\"{o}lder continuous, the behavior of
$S_n(f)$ is very well understood, thanks to Young (1999) and
Melbourne-Nicol (2005): these sums satisfy the almost sure
invariance principle for $\gamma<1/2$ (in particular, the
central limit theorem and the law of the iterated logarithm
hold). For the Liverani-Saussol-Vaienti map, Gou\"{e}zel (2004a)
shows that, when $\gamma\in (1/2,1)$ and $f$ is Lipschitz
continuous, $S_n(f)$ suitably renormalized converges to a
gaussian law (resp.~a stable law) if $f(0)=\nu(f)$ (resp.
$f(0)\not=\nu(f)$).

On the other hand, when $f$ is less regular, much less is
known. If $f$ has finitely many discontinuities and is
otherwise H\"{o}lder continuous, the construction of Young (1999)
could be adapted to obtain a tower avoiding the discontinuities
of $f$ -- the almost sure invariance principle follows when
$\gamma<1/2$. However, functions with countably many
discontinuities are not easily amenable to the tower method,
and neither are very simple unbounded functions such as
$g(x)=\ln |x-x_0|$ or $g_a(x)=|x-x_0|^a$ for any $x_0\not=0$.
This is far less satisfactory than the i.i.d.~situation, where
optimal moment conditions for the invariance principle or the
central limit theorem are known, and it seems especially
interesting to devise new methods than can handle functions
under moment conditions as close to the optimum as possible.

For the Liverani-Saussol-Vaienti maps, using martingale
techniques, Dedecker and Prieur (2009) proved that the central limit
theorem holds for a much larger class of functions (including all the
functions of bounded variation and several piecewise monotonic
unbounded discontinuous functions, for instance the functions
$g$ and $g_a$ above up to the optimal value of $a$) -- our
arguments below show that their results in fact hold for all GPM
maps, not only markovian ones. Our main goal in this article is
to prove the bounded law of the iterated logarithm for the same
class of functions. We shall also make  use of martingale techniques,
but we will also need a  more precise control on the
behavior of bounded variation functions under the iteration of
GPM maps.

\smallskip

The main steps of our approach are the following:
\begin{enumerate}
\item \emph{The main probabilistic tool.}
Let $(Y_1,Y_2,\dots)$ be an arbitrary stationary process. We describe in Paragraph
\ref{parproba} a coefficient $\alpha$ which measures (in a weak way) the asymptotic
independence in this process, and was introduced in Rio (2000). It is weaker than the
usual mixing coefficient of Rosenblatt (1956), since it only involves events of the
form $\{Y_i \leq x_i\}$, $x_i\in \R$. In particular, it can tend to $0$ for some
processes that are not Rosenblatt mixing (this will be the case for the processes to
be studied below). Thanks to its definition, $\alpha$ behaves well under the
composition with monotonic maps of the real line. This coefficient $\alpha$ contains
enough information to prove the maximal inequality stated in Proposition \ref{FN}, by
following the approach of Merlev\`{e}de (2008). In turn, this inequality implies (a
statement more precise than) the bounded law of the iterated logarithm  given in
Theorem \ref{ASthm}, for processes of the form $(f(Y_1),f(Y_2),\dots)$ where
$(Y_1,Y_2,\dots)$ has a well behaved $\alpha$ coefficient, and $f$ belongs to a large
class of functions.
\item \emph{The main dynamical tool.}
Let $K$ denote the Perron-Frobenius operator of $T$ with
respect to $\nu$, given by
  \begin{equation}
  \label{def_PF}
  K f(x)=\frac{1}{h(x)}\sum_{T(y)=x} \frac{h(y)}{|T'(y)|}f(y),
  \end{equation}
where $h$ is the density of $\nu$. For any bounded
measurable functions $f$, $g$, it satisfies $\nu(f \cdot
g\circ T)=\nu(K(f) g)$. Since $\nu$ is invariant by $T$, one has
$K(1)=1$, so that $K$ is a Markov operator.
Following the approach of Gou\"{e}zel (2007),  we will study
the operator $K$ on the space $\BV$ of bounded variation
functions, show that its iterates are uniformly bounded,
and estimate the contraction of $K^n$ from $\BV$ to $L^1$
(in Propositions \ref{propvar} and \ref{propdecaycorr}).

\item Let us denote by $(Y_i)_{i \geq 1}$ a stationary Markov
chain with invariant measure $\nu$ and transition kernel
$K$. Since the mixing coefficient $\alpha$ involves events
of the form $\{Y_i \leq x_i\}$, it can be read from the
behavior of $K$ on $\BV$. Therefore, the previous estimates
yield a precise control of the coefficient $\alpha$ of this
process. With Theorem \ref{ASthm}, this gives a bounded law
of the iterated logarithm for the process
$(f(Y_1),f(Y_2),\dots)$.

\item
It is well known that on the probability space $([0, 1],
\nu)$, the random variable $(f, f\circ T, \ldots , f\circ
T^{n-1})$ is distributed as $(f(Y_n),f(Y_{n-1}), \ldots,
f(Y_1))$. Since there is a phenomenon of time reversal, the
law of the iterated logarithm for $(f(Y_1),f(Y_2),\dots)$
does not imply the same result for $(f, f\circ T,\dots)$.
However, the technical statement of Theorem \ref{ASthm} is
essentially invariant under time reversal, and therefore
also gives a bounded law of the iterated logarithm for
$S_n(f)$.
\end{enumerate}

In the next three paragraphs, we describe our results more
precisely. The proofs are given in the remaining sections.

\begin{rmkm}
The class of maps covered by our results could be further
extended, as follows. First, we could allow finitely many
neutral fixed point, instead of a single one (possibly with
different behaviors). Second, we could allow infinitely many
monotonicity branches for $T$ if, away from the neutral fixed
points, the quantity $|T''|/(T')^2$ remains bounded, and the
set $\{T(Z)\}$, for $Z$ a monotonicity interval, is finite
(this is for instance satisfied if all branches but finitely
many are onto). Finally, we could drop the topological
transitivity.

The ergodic properties of this larger class of maps is fully
understood thanks to the work of Zweim\"{u}ller (1998): there are
finitely many invariant measures instead of a single one, and
the support of each of these measures is a finite union of
intervals. Our arguments still apply in this broader context,
although notations and statements become more involved. For the
sake of simplicity, we shall only consider the class of GPM
maps (which is already quite large).
\end{rmkm}

\subsection{Statements of the results for intermittent
maps} \label{parglob}

\begin{defn}
A function $H$ from ${\mathbb R}_+$ to $[0, 1]$ is a tail
function if it is non-increasing, right continuous, converges
to zero at infinity, and $x\rightarrow x H(x)$ is integrable.
\end{defn}
\begin{defn}
\label{defMon}
If $\mu$ is a probability measure on $\mathbb R$ and $H$ is a
tail function, let $\Mon(H, \mu)$ denote the set of functions
$f:\R\to \R$ which are monotonic on some open interval and null
elsewhere and such that $\mu(|f|>t)\leq H(t)$. Let $\Monm(H,
\mu)$ be the closure in ${\mathbb L}^1(\mu)$ of the set of
functions which can be written as $\sum_{\ell=1}^L a_\ell
f_\ell$, where $\sum_{\ell=1}^L |a_\ell| \leq 1$ and $f_\ell\in
\Mon(H, \mu)$.
\end{defn}

Note that a function belonging to $\Monm(H, \mu)$ is allowed to
blow up at an infinite number of points. Note also that any
function $f$ with bounded variation ($\BV$) such that $|f|\leq
M_1$ and $\|df\|\leq M_2$ belongs to the class $\Monm(H, \mu)$
for any $\mu$ and the tail function $H=\I_{[0, M_1+2M_2)}$
(here and henceforth, $\|df\|$ denotes the variation norm of
the signed measure $df$). Moreover, if a function $f$ is
piecewise monotonic with $N$ branches, then it belongs to
$\Monm(H, \mu)$ for $H(t)=\mu(|f|>t/N)$. Finally, let us
emphasize that there is no requirement on the modulus of
continuity for functions in $\Monm(H, \mu)$

Our first result is a bounded law of the iterated logarithm,
when $0<\gamma< 1/2$.

\begin{thm} \label{ASmap} Let $T$ be a GPM map
with parameter $\gamma\in (0,1/2)$ and invariant measure $\nu$.
Let $H$ be a tail function with
\begin{equation}\label{lilcond}
\int_0^{\infty} x (H(x))^{\frac{1-2\gamma}{1-\gamma}} dx <\infty \,.
\end{equation}
Then, for any $f \in \Monm(H, \nu)$, the series
\[\sigma^2= \nu((f-\nu(f))^2)+ 2 \sum_{k>0}
\nu ((f-\nu(f))f\circ T^k)
\]
converges absolutely to some nonnegative number. Moreover,
\begin{enumerate}
\item There exists a nonnegative constant $A$ such that
\begin{equation}\label{lilbounded}
  \sum_{n=1}^\infty \frac{1}{n} \nu \Big( \max_{1 \leq k \leq
  n} \Big |\sum_{i=0}^{k-1} (f \circ T^i
  -\nu(f))\Big|\geq A \sqrt {n \ln (\ln (n))} \Big) < \infty \, ,
\end{equation}
and consequently\footnote{see e.g.~Stout (1974), Chapter 5.}
\begin{equation*}
\limsup_{n \rightarrow \infty} \frac{1}{\sqrt {n \ln (\ln (n))}}
\Big|\sum_{i=0}^{n-1} (f\circ T^i - \nu(f)) \Big | \leq A
\, , \ \text{almost everywhere.}
\end{equation*}
\item Let $(Y_i)_{i \geq 1}$ be a stationary Markov chain with
transition kernel $K$ and invariant measure $\nu$, and let
$X_i=f(Y_i)-\nu(f)$. Enlarging if necessary the underlying
probability space, there exists a sequence $(Z_i)_{i \geq
1}$ of i.i.d.~gaussian random variables with mean zero and
variance $\sigma^2$ such that
\begin{equation}\label{ASIP}
\Big|\sum_{i=1}^n (X_i- Z_i) \Big| =o(\sqrt {n \ln (\ln (n))})\, ,
  \text{almost surely.}
\end{equation}
\end{enumerate}
\end{thm}

In particular, we infer that the bounded law \eqref{lilbounded}
holds for any $\BV$ function $f$ provided that $\gamma<1/2$.
Note also that \eqref{lilcond} is satisfied provided that
$H(x)\leq C x^{-2(1-\gamma)/(1-2\gamma)}(\ln(x))^{-b}$ for $x$
large enough and $b>(1-\gamma)/(1-2\gamma)$. Let us consider
two simple examples. Since the density $h_{\nu}$ of $\nu$ is
such that $h_{\nu}(x) \leq C x^{-\gamma}$ on $(0, 1]$, one can
easily prove that:
\begin{enumerate}
\item If $f$ is positive and non increasing on (0, 1), with
\[
f(x) \leq \frac{C}{x^{(1-2\gamma)/2}|\ln(x)|^{b}}
\quad \text{near 0, for some $b>1/2$,}
\]
then  \eqref{lilbounded}  and \eqref{ASIP} hold.
\item If $f$ is positive and non decreasing on (0, 1), with
\[
f(x) \leq \frac{C}{(1-x)^{(1-2\gamma)/(2-2\gamma)}|\ln(1-x)|^{b}}\quad
\text{near 1, for some  $b>1/2$,}
\]
then
\eqref{lilbounded}  and \eqref{ASIP} hold.
\end{enumerate}

In fact, if $f \in \Monm(H, \nu)$ for some $H$ satisfying
\eqref{lilcond} then the central limit theorem and the weak
invariance principle hold. This can be easily deduced from the
proof of Theorem 4.1 in Dedecker and Prieur (2009) and by using
the upper bound for the coefficient $\alpha_{1, {\bf Y}} (k)$
given in Proposition \ref{weakalpha} (which improves on the
corresponding bound in Dedecker and Prieur (2009)). Hence, if
$f$ is as in Item 1 above, both the central limit theorem and
the bounded law of the iterated logarithm hold.

An open question is: can we obtain the almost sure invariance
principle \eqref{ASIP} for the sequence $(f\circ T^i)_{i \geq
0}$ instead of $(f(Y_i))_{i \geq 1}$? According to the
discussion in Melbourne and Nicol (2005), this appears to be a
rather delicate question. Indeed, to obtain Item 2 of Theorem
\ref{ASmap}, we use first a maximal inequality for the partial
sums $\sum_{i=1}^k f(Y_i)$ and next a result by Voln\'{y} and Samek
(2000) on the approximating martingale. As pointed out by
Melbourne and Nicol (2005, Remark 1.1), we cannot go back to
the sequence $(f\circ T^i)_{i \geq 0}$, because the system is
not closed under time reversal. Using another approach, going
back to Philipp and Stout (1975) and Hofbauer and Keller
(1982), Melbourne and Nicol (2005) have proved the almost sure
invariance principle for $(f\circ T^i)_{i \geq 0}$ when
$\gamma<1/2$ and $f$ is any H\"{o}lder continuous function, with a
better error bound $O(n^{1/2-\epsilon})$ for some $\epsilon>0$.
As a consequence, their result imply the functional law of the iterated logarithm for H\"{o}lder continuous function,
which is much more precise than the bounded law.
However, our approach is clearly distinct from that of Melbourne and Nicol (2005), for we
cannot deduce the control \eqref{lilbounded} from an almost
sure invariance principle.

\medskip

In the next theorem, we give rates of convergence in the strong law of large numbers under
weaker conditions than (\ref{lilcond}), which do not imply the  central limit theorem.

\begin{thm} \label{LLNmap} Let $1<p<2$ and $0<\gamma< 1/p$. Let
$T$ be a GPM map with parameter $\gamma$ and invariant measure
$\nu$. Let $H$ be a tail function with
\begin{equation}\label{ratecond}
 \int_0^{\infty} x^{p-1} (H(x))^{\frac{1-p\gamma}{1-\gamma}} dx < \infty \,.
\end{equation}
Then, for any $f \in \Monm(H, \nu)$ and any $\varepsilon>0$,
one has
\begin{equation}\label{SLN}
  \sum_{n=1}^\infty \frac{1}{n} \nu \Big( \max_{1 \leq k \leq
  n} \Big |\sum_{i=1}^k (f \circ T^i
  -\nu(f))\Big|\geq n^{1/p} \varepsilon \Big) < \infty \, .
\end{equation}
Consequently, $n^{-1/p}\sum_{k=1}^n (f \circ T^i -\nu(f))$
converges to $0$ almost everywhere.
\end{thm}
Note that \eqref{ratecond} is satisfied provided that $H(x)\leq C
x^{-p(1-\gamma)/(1-p\gamma)}(\ln(x))^{-b}$ for $x$ large enough and
$b>(1-\gamma)/(1-p\gamma)$. For instance, one can easily prove that, for  $1<p < 2$ and
$0<\gamma< 1/p$,
\begin{enumerate}
\item If $f$ is positive and non increasing on $(0, 1)$, with
\[f(x) \leq \frac{C}{x^{(1-p\gamma)/p}|\ln(x)|^{b}}
\quad \text{near 0, for some $b>1/p$,}\] then
\eqref{SLN} holds.
\item If $f$ is positive and non decreasing on $(0, 1)$, with
\[f(x) \leq \frac{C}{(1-x)^{(1-p\gamma)/(p-p\gamma)}|\ln(1-x)|^{b}}\quad  \text{near 1, for some  $b>1/p$,}\]
 then  \eqref{SLN} holds.
\end{enumerate}

The condition (\ref{ratecond}) of Theorem \ref{LLNmap} means exactly that the probability
$\mu_{H, p, \gamma}$ on ${\mathbb
 R}_+$ such that $ \mu_{H, p,
\gamma}((x, \infty))=(H(x))^{\frac{1-p\gamma}{1-\gamma}}$ has a moment of order $p$. Let
us see what happen if we only assume that $\mu_{H, p, \gamma}$ has  a weak moment of
order $p$.

\begin{thm}\label{LLNmap2}
Let $1<p\leq 2$ and $0<\gamma< 1/p$. Let $T$ be a GPM map with
parameter $\gamma$ and invariant measure $\nu$. Let $H$ be a
tail function with
\begin{equation}\label{ratecond*}
 (H(x))^{\frac{1-p\gamma}{1-\gamma}} \leq C x^{-p}  \,.
\end{equation}
Then, for any $f \in \Monm(H, \nu)$, any $b>1/p$ and  any
$\varepsilon>0$, one has
\begin{equation}\label{SLN2}
  \sum_{n=1}^\infty \frac{1}{n} \nu \Big( \max_{1 \leq k \leq
  n} \Big |\sum_{i=0}^{k-1} (f \circ T^i
  -\nu(f))\Big|\geq n^{1/p}(\ln(n))^{b} \varepsilon \Big) < \infty \, .
\end{equation}
Consequently, $n^{-1/p}(\ln(n))^{-b}\sum_{k=0}^{n-1} (f \circ
T^i -\nu(f))$ converges to $0$ almost everywhere.
\end{thm}

Applying Theorem \ref{LLNmap2}, one can easily prove that, for $1<p\leq 2$ and $0<\gamma<
1/p$,
\begin{enumerate}
\item If $f$ is positive and non increasing on $(0, 1)$, with
$f(x) \leq Cx^{-(1-p\gamma)/p} $ then \eqref{SLN2} holds.
\item If $f$ is positive and non decreasing on $(0, 1)$, with
$f(x) \leq C(1-x)^{-(1-p\gamma)/(p-p\gamma)}$ then
\eqref{SLN2} holds.
\end{enumerate}

This requires additional comments.  Gou\"{e}zel (2004a) proved that
if $f$ is exactly of the form $f(x)=x^{-(1-p\gamma)/p}$ for $
1<p <2$ and $0< \gamma <1/p$, then $n^{-1/p}\sum_{i=0}^{n-1} (f
\circ T^i -\nu(f))$ converges in distribution on $([0, 1],
\nu)$ to a centered one-sided stable law of index $p$, that is
a stable law whose distribution function $F^{(p)}$ is such that
$x^pF^{(p)}(-x) \rightarrow 0$ and $x^p(1-F^{(p)}(x))
\rightarrow c$,  as $x \rightarrow \infty$, with $c>0$. Our
theorem shows that $n^{-1/p}(\ln(n))^{-b} (\sum_{i=1}^n (f
\circ T^i -\nu(f)))$ converges almost everywhere to zero for
$b>1/p$. This is in total accordance with the i.i.d.~situation,
as we describe now. Let $(X_i)_{i \geq 1}$ be a sequence of
i.i.d.~centered random variables satisfying $n^{-1/p} (X_1+
\cdots + X_n) \to F^{(p)}$. It is well known (see for instance
Feller (1966), page 547) that this is equivalent to $x^p
{\mathbb P}(X_1< -x)\rightarrow 0$ and $x^p {\mathbb
P}(X_1>x)\rightarrow c$ as $x \rightarrow \infty$. For any
nondecreasing sequence $(b_n)_{n \geq 1}$ of positive numbers,
either $(X_1+ \cdots + X_n)/b_n $ converges to zero almost
surely or $\limsup_{n \rightarrow \infty} |X_1+ \cdots +
X_n|/b_n=\infty$ almost surely, according as $\sum_{n=1}^\infty
{\mathbb P}(|X_1|>b_n)< \infty$ or $\sum_{n=1}^\infty {\mathbb
P}(|X_1|>b_n)= \infty$ -- this follows from the proof of
Theorem 3 in Heyde (1969). If one takes $b_n=n^{1/p}(\ln(n))^b$
we obtain the constraint $b>1/p$ for the almost sure
convergence of $n^{-1/p}(\ln(n))^{-b} (X_1+ \cdots +X_n)$ to
zero. This is exactly the same constraint as in our dynamical
situation.

Let us comment now on the case $p=2$. In his (2004a) paper,
Gou\"{e}zel also proved that if $f$ is exactly of the form
$f(x)=x^{-(1-2\gamma)/2}$ then the central limit theorem holds
with the normalization $\sqrt{n \ln (n)}$. As mentioned above
such an  $f$ belongs to the class $\Monm(H, \nu)$ for some $H$
satisfying (\ref{ratecond*}) with $p=2$, which means that
$\mu_{H, 2, \gamma}$ has a weak moment of order 2. This again
is in accordance with the i.i.d.~situation.  Let $(X_i)_{i \geq
1}$ be a sequence of i.i.d.~centered random variables such that
$x^2{\mathbb P}(X_1<-x )\rightarrow c_1$ and $x^2{\mathbb
P}(X_1>x)\rightarrow c_2$ as $x$ tends to infinity, with
$c_1+c_2=1$. Then $(n \ln(n))^{-1/2}(X_1+ \cdots + X_n)$
converges in distribution to a standard gaussian distribution,
but according to Theorem 1 in Feller (1968),
$$
 \limsup_{n \rightarrow \infty} \frac{1}{\sqrt {n \ln(n) \ln(\ln(n))}}\sum_{i=1}^n X_i= \infty \, .
$$
Moreover, if $(b_n)_{n \geq 1}$ is a non decreasing sequence such that $b_n/\sqrt {n
\ln(n) \ln(\ln(n))}\rightarrow \infty$ (plus the mild conditions (2.1) and (2.2) in
Feller's paper), then either $(X_1+ \cdots + X_n)/b_n $ converges to zero almost surely
or $\limsup_{n \rightarrow \infty} |X_1+ \cdots + X_n|/b_n=\infty$ almost surely,
according as $\sum_{n=1}^\infty {\mathbb P}(|X_1|>b_n)< \infty$ or $\sum_{n=1}^\infty
{\mathbb P}(|X_1|>b_n)= \infty$. If one takes $b_n=n^{1/2}(\ln(n))^b$ we obtain the
constraint $b>1/2$ for the almost sure convergence of $n^{-1/2}(\ln(n))^{-b} (X_1+ \cdots
+X_n)$ to zero. This is exactly the same constraint as in our dynamical situation.

\subsection{A general result for stationary sequences}
\label{parproba}

Before stating the maximal inequality proved in this paper, we
shall introduce some definitions and notations.

\begin{defn} \label{defquant}
 For any nonnegative random variable
$X$, define the ``upper tail'' quantile function $Q_X $ by $
Q_X (u) = \inf \left \{  t \geq 0 : \p \left(X >t \right) \leq
u\right \} $.
\end{defn}
This function is defined on $[0,1]$, non-increasing, right
continuous, and has the same distribution as $X$. This makes it
very convenient to express the tail properties of $X$ using
$Q_X$. For instance, for $0<\varepsilon<1$, if the distribution
of $X$ has no atom at $Q_X(\varepsilon)$, then
  \begin{equation*}
  \E ( X \I_{X > Q_X(\varepsilon)})=\sup_{\p(A)\leq \varepsilon} \E(X \I_A) = \int_0^\varepsilon Q_X(u) du\, .
  \end{equation*}

\begin{defn}
\label{defclosedenv} Let $\mu$ be the probability distribution of a
random variable $X$. If $Q$ is an integrable quantile function,
let $\tMon(Q, \mu)$ be the set of functions $g$ which are
monotonic on some open interval of ${\mathbb R}$ and null
elsewhere and such that $Q_{|g(X)|} \leq Q$. Let $\tMonm(Q,
\mu)$ be the closure in ${\mathbb L}^1(\mu)$ of the set of
functions which can be written as $\sum_{\ell=1}^{L} a_\ell
f_\ell$, where $\sum_{\ell=1}^{L} |a_\ell| \leq 1$ and $f_\ell$
belongs to $\tMon(Q, \mu)$.
\end{defn}
This definition is similar to Definition \ref{defMon}, we only
use quantile functions instead of tail functions. There is in
fact a complete equivalence between these two points of view:
if $Q$ is a quantile function and $H$ is its c\`{a}dl\`{a}g inverse,
then $\tMon(Q, \mu)=\Mon(H,\mu)$ and
$\tMonm(Q,\mu)=\Monm(H,\mu)$.

Let now  $(\Omega ,\mathcal{A}, \p)$ be a probability space, and let
$\theta :\Omega \mapsto \Omega $ be a bijective bimeasurable
transformation preserving the probability ${\p}$. Let
${\mathcal M}_0$ be a sub-$\sigma$-algebra of $\mathcal{A}$
satisfying ${\mathcal M}_0 \subseteq \theta^{-1}({\mathcal
M}_0)$.

\begin{defn}
For any integrable random variable $X$, let us write
$X^{(0)}=X- \E(X)$.
For any random variable $Y=(Y_1, \cdots, Y_k)$ with values in
${\mathbb R}^k$ and any $\sigma$-algebra $\F$, let
\[
\alpha(\F, Y)= \sup_{(x_1, \ldots , x_k) \in {\mathbb R}^k}
\left \| \E \Big(\prod_{j=1}^k (\I_{Y_j \leq x_j})^{(0)} \Big | \F \Big)^{(0)} \right\|_1.
\]
For a sequence ${\bf Y}=(Y_i)_{i \in {\mathbb Z}}$, where
$Y_i=Y_0 \circ \theta^i$ and $Y_0$ is an $\M_0$-measurable and
real-valued random variable, let \begin{equation}
\label{defalpha} \alpha_{k, {\bf Y}}(n) = \max_{1 \leq l \leq
k} \ \sup_{ n\leq i_1\leq \ldots \leq i_l} \alpha(\M_0,
(Y_{i_1}, \ldots, Y_{i_l})) .
\end{equation}
\end{defn}

The following maximal inequality is crucial for the
proof of Theorem \ref{ASthm} below.

\begin{prop} \label{FN}
Let  $X_i = f(Y_i) - \E ( f(Y_i))$, where $Y_i=Y_0 \circ
\theta^i$ and $f$ belongs to $\tMonm(Q, P_{Y_0})$ (here,
$P_{Y_0}$ denotes the distribution of $Y_0$, and $Q$ is a square
integrable quantile function). Define the coefficients
$\alpha_{1, {\bf Y}}(n) $ and $\alpha_{2, {\bf Y}}(n) $ as in
\eqref{defalpha}. Let $n\in\N$. Let
  \begin{equation*}
  R(u)=(\min\{q\in \N \tq \alpha_{2,{\bf Y}}(q)\leq u\} \wedge n) Q(u)
  \text{ and }
  S(v)=R^{-1}(v)=\inf\{u\in [0,1] \tq R(u)\leq v\}\,.
  \end{equation*}
Let $S_n = \sum_{k=1}^n X_k$. For any $x>0$, $r \geq 1$, and
$s_n>0$ with $s_n^2 \geq 4n \sum_{i=0}^{n-1} \int_0^{\alpha_{1,
{\bf Y}} (i)}Q^2 (u) du$, one has
  \begin{equation}
  \label{eqmainmax}
  \p \Big( \sup_{ 1 \leq k \leq n}  |
  S_k | \geq 5x \Big)  \leq  4 \exp \left ( - \frac{r^2 s_n^2}{8x^2}
  h \left( \frac{ 2 x^2 }{ r  s_n^2}\right) \right
  )+  n \Big ( \frac{6}{x} +
  \frac{16x }{ r s_n^2} \Big ) \int_0^{S(x/r)}Q(u) du \, ,
  \end{equation}
where $h(u):= (1+u) \ln (1+u) -u$.
\end{prop}
\begin{rmkm} Note that a similar bound for $\alpha$-mixing sequences
in the sense of Rosenblatt (1956) has been proved in Merlev\`{e}de (2008, Theorem 1).
Since $h(u)\geq u \ln (1+u) /2$, under the notation and
assumptions of the above theorem, we get that for any $x
>0$ and $r \geq 1$,
\begin{eqnarray} \label{r1FN}
 \p \Big( \sup_{ 1 \leq k \leq n}  |
S_k | \geq 5x \Big) & \leq & 4 \left ( 1 + \frac{ 2 x^2  }{ r
s_n^2} \right )^{-r/8}+  n \Big ( \frac{6}{x} + \frac{16x }{ r
s_n^2} \Big ) \int_0^{S(x/r)}Q (u) du \, .
\end{eqnarray}
\end{rmkm}

Theorem \ref{ASmap} is in fact a corollary of the following
theorem, which gives both a precise control of the tail of the
partial sums by applying Proposition \ref{FN}, and a strong
invariance principle for the partial sums.

Let ${\mathcal I}$ be the $\sigma$-algebra of all $\theta$-invariant sets.
The map $\theta$ is ${\mathbb P}$-ergodic if each element of ${\mathcal I}$ has  measure $0$ or $1$.

\begin{thm} \label{ASthm}
Let $Y_i$, $X_i$ and $S_n$ be as in Proposition \ref{FN}. Assume that the following
condition is satisfied:
\begin{equation} \label{condQ} \sum_{k \geq 1} \int_0^{\alpha_{2, {\bf Y}}(k)} Q^2(u) du
< \infty \, . \end{equation} Then the series $\sigma^2=\sum_{k
\in {\mathbb Z}}\Cov(X_0, X_k)$ converges absolutely to some
nonnegative number $\sigma^2$, and
\begin{eqnarray} \label{r1app2LIL}
\sum_{  n >0}  \frac 1n {\mathbb P} \Big ( \sup_{k \in [1,n]} |S_k |
\geq A \sqrt{2 n \ln (\ln (n))} \Big ) < \infty   ,  \text{with} \
A = 20 \Big (\sum_{k \geq 0} \int_0^{\alpha_{1, {\bf Y}}(k)} Q^2(u) du \Big
)^{1/2}\, .
\end{eqnarray}
Assume moreover that $\theta$ is ${\mathbb P}$-ergodic. Then, enlarging  $\Omega$ if necessary, there exists a sequence
$(Z_i)_{i \geq 0}$ of i.i.d.~gaussian random variables with
mean zero and variance $\sigma^2$ such that
\begin{eqnarray} \label{asr1}
  \Big|S_n - \sum_{i=1}^n Z_i \Big| =o(\sqrt {n \ln(\ln( n))}) , \
  \text{almost surely.}
\end{eqnarray}
\end{thm}
\begin{rmkm} The  strong invariance principle for $\alpha$-mixing sequences (in the sense
of Rosenblatt (1956))  given in  Rio (1995) Theorem 2,
 can be easily deduced from \eqref{asr1}. Note that the
optimality of Rio's result is discussed in Theorem 3 of his
paper.
\end{rmkm}

\subsection{Dependence coefficients for intermittent maps}
\label{pardyn}

Let $\theta$ be the shift operator from ${\mathbb R}^{\mathbb
Z}$ to ${\mathbb R}^{\mathbb Z}$ defined by
$(\theta(x))_i=x_{i+1}$, and let $\pi_i$ be the projection from
${\mathbb R}^{\mathbb Z}$ to ${\mathbb R}$ defined by
$\pi_i(x)=x_i$. Let ${\bf Y}=(Y_i)_{i\geq 0}$ be a stationary
real-valued Markov chain with transition kernel $K$ and
invariant measure $\nu$. By Kolmogorov's extension theorem,
there exists a shift-invariant probability ${\mathbb P}$ on
$({\mathbb R}^{\mathbb Z}, ({\mathcal B}({\mathbb R}))^{\mathbb
Z})$, such that $\pi=(\pi_i)_{i\geq 0}$ is distributed as ${\bf
Y}$. Let ${\mathcal M}_0= \sigma (\pi_i, i \leq 0)$. We define
the coefficient $\alpha_{k,{\bf Y}}(n)$ of the chain
$(Y_i)_{i\geq 0}$ {\it via} its extension $(\pi_i)_{i\in
{\mathbb Z}}$: $ \alpha_{k,{\bf Y}}(n)=\alpha_{k,{\bf
\pi}}(n)$.

Note that these coefficients may be written in terms of the
kernel $K$ as follows. Let $f^{(0)}=f-\nu(f)$. For any
non-negative integers $n_1, n_2, \ldots , n_k$, and any bounded
measurable functions $f_1, f_2, \ldots, f_k$, define
  \begin{equation*}
  K^{(0)(n_1, n_2, \ldots, n_k)}(f_1, f_2, \ldots, f_k)=\big(K^{n_1}(f_1
  K^{n_2}(f_2 K^{n_3}(f_3\cdots K^{n_{k-1}}(f_{k-1}K^{n_k}(f_k))\cdots
  )))\big)^{(0)}\, .
  \end{equation*}
Let $\BV_1$ be space of bounded variation functions $f$ such
that $\|df\|\leq 1$, where $\|df\|$ is  the variation norm on $\R$ of the measure
$df$. We have
  \begin{equation}
  \label{defequiv}
  \alpha_{k, {\bf Y}}(n)=\sup_{1\leq l \leq k}\sup_{n_1\geq n, n_2\geq 0,
  \ldots n_l\geq 0}\sup_{f_1, \ldots, f_l \in {\BV_1}}
  \nu\big(|K^{(0)(n_1, n_2, \ldots, n_l)}(f^{(0)}_{1}, f^{(0)}_{2}, \ldots,
  f^{(0)}_{l})|\big)\, .
  \end{equation}

Let us now fix a GPM map $T$ of parameter $\gamma\in (0,1)$.
Denote by $\nu$ its absolutely continuous invariant probability
measure, and by $K$ its Perron-Frobenius operator with respect
to $\nu$. Let ${\bf Y}=(Y_i)_{i\geq 0}$ be a stationary Markov
chain with invariant measure $\nu$ and transition kernel $K$.

The following proposition shows that the iterates of $K$ on
$\BV$ are uniformly bounded.

\begin{prop}\label{propvar}
There exists  $C>0$, not depending on $n$, such that for any
$\BV$ function $f$, $\|dK^n (f)\| \leq C \|df\|$.
\end{prop}

The following covariance inequality implies an estimate on
$\alpha_{1, {\bf Y}}$.

\begin{prop}
\label{propdecaycorr} There exists $B>0$ such that, for any bounded function $\varphi$, any
$\BV$ function $f$ and any $n>0$
  \begin{equation}
  \label{controleiter}
  |\nu(\varphi \circ T^n \cdot (f-\nu(f)))|
  \leq \frac{B}{n^{(1-\gamma)/\gamma}}\|df\| \|\varphi\|_\infty \, .
  \end{equation}
\end{prop}

Putting together the last two propositions and
\eqref{defequiv}, we obtain the following:
\begin{prop}
\label{weakalpha}
For any positive integer $k$, there exists a constant $C$ such
that, for any $n>0$,
\[
  \alpha_{k,{\bf Y}}(n) \leq
  \frac{C}{n^{(1-\gamma)/\gamma}} \, .
\]
\end{prop}
\begin{proof} Let $f \in \BV_1$ and $g \in \BV$ with $\|g\|_\infty \leq 1$. Then, applying Proposition \ref{propvar}, we obtain for any $n\geq 0$,
\begin{equation}\label{bla}
  \|d(f^{(0)}K^n(g))\|\leq \|df\|\|g\|_\infty + \|dK^n(g)\|\|f^{(0)}\|_\infty \leq 1+ C \|dg\|\, .
\end{equation}
For $f_1,\dots,f_k\in \BV_1$,  let $f=f_1^{(0)}
K^{n_2}(f_2^{(0)} K^{n_3}(f_3^{(0)}\cdots
K^{n_{k-1}}(f_{k-1}^{(0)}K^{n_k}(f_k^{(0)}))\cdots )$.
Iterating  Inequality (\ref{bla}), we obtain, for any   $n_2,\dots,n_k\geq
0$,
$
  \|df\| \leq 1+C+ C^2 + \cdots + C^{k-1}
$.
Together with the bound \eqref{defequiv} for $\alpha_{k, {\bf
Y}}(n)$, this implies that
$$
  \alpha_{k, {\bf
Y}}(n) \leq (1+C+ C^2 + \cdots + C^{k-1}) \alpha_{1, {\bf
Y}}(n) \, .
$$
Now the upper bound (\ref{controleiter}) means exactly that $\alpha_{1, {\bf
Y}}(n) \leq B n^{(\gamma-1)/\gamma}$, which concludes the proof of Proposition (\ref{weakalpha}).
\end{proof}


Proposition \ref{weakalpha} improves on the corresponding upper
bound given in Dedecker and Prieur (2009). Let us mention that
this upper bound is optimal: the lower bound $\alpha_{k,{\bf
Y}}(n) \geq C'n^{(\gamma-1)/\gamma}$ was given in
Dedecker and Prieur (2009) for Liverani-Saussol-Vaienti maps,
and is a consequence in this markovian context of the lower
bound for $\nu(\varphi \circ T^n \cdot (f-\nu(f)))$ given by
Sarig (2002), Corollary 1. Our techniques imply that this lower
bound also holds in the general setting of GPM maps.

\medskip

In the rest of the paper, we prove the previous results. First,
in Section \ref{secproba}, we prove the results of Paragraph
\ref{parproba}, which are essentially of probabilistic nature.
In Section \ref{secdyn}, we study the transfer operator of a
GPM map $T$, to prove the dynamical results of Paragraph
\ref{pardyn}. Finally, in the last section, we put together all
those results (and arguments of Dedecker and Merlev\`{e}de (2007))
to prove the main theorems of Paragraph \ref{parglob}.

In the rest of this paper, $C$ and $D$ are positive constants
that may vary from line to line.

\section{Proofs of the probabilistic results}
\label{secproba}

\subsection{Proof of Proposition \ref{FN}}
Assume first that $X_i = \sum_{\ell=1}^L a_{\ell} f_{\ell}(Y_i)
- \sum_{\ell=1}^L a_{\ell}\E(f_{\ell}(Y_i))$, with $f_\ell$
belonging to $\tMon(Q, P_{Y_0})$ and $\sum_{\ell=1}^L |a_\ell|
\leq 1$. Let $M>0$ and $g_M(x) = (x \wedge M) \vee (-M)$. For
any $i \geq 0$, we first define
\[
 X_i'=\sum_{\ell=1}^L a_{\ell} \, g_M \circ f_{\ell}(Y_i) -\sum_{\ell=1}^L a_{\ell}
  \E ( g_M \circ f_{\ell}(Y_i)) \quad \text{and} \quad  X_i''=X_i - X_i' \, .
\]

Let $S_n'= \sum_{i=1}^n X_i'$ and $S_n''= \sum_{i=1}^n X_i''$.
Let $q$ be a positive integer and for $1\leq i \leq[n/q]$,
define the random variables $U'_i = S'_{iq}- S'_{iq -q}$ and
$U''_i = S''_{iq}- S''_{iq -q}$.

Let us first show that
  \begin{equation} \label{dec1FN}
  \max_{1 \leq k \leq n } |S_k | \leq
  \max_{ 1 \leq j \leq [n/q]}  \Big|\sum_{i=1}^j U_i' \Big|
  + 2 q M + \sum_{k=1}^n |X_k''|  \, .
  \end{equation}
If the maximum of $|S_k |$ is obtained for $k=k_0$, then for
$j_{0} =[k_0/q]$,
  \begin{equation*}
  \max_{1 \leq k \leq  n } | S_k
  | \leq \Big| \sum_{i=1}^{j_0} U_i' \Big| + \sum_{i=1}^{j_0} |
  U_i'' | + \sum_{k= q j_0 + 1}^{k_0} |X'_k |+ \sum_{k= q j_0 + 1}^{k_0} |X''_k |\, .
  \end{equation*}
Since $|X'_k|\leq 2M \sum_{\ell=1}^L |a_{\ell}|  \leq 2M$, and
$\sum_{i=1}^{j_0} | U_i'' | \leq  \sum_{k=1}^{q j_0} | X''_k
|$, this concludes the proof of \eqref{dec1FN}.

\medskip

For all $i \geq 1$, let $\F^U_{i}= {\mathcal M}_{iq}$, where
${\mathcal M}_k= \theta^{-k}({\mathcal M}_0)$. We define a
sequence $(\tilde U_i)_{i \geq 1}$ by $\tilde U_i=U'_i-\E(U'_i
| \F^U_{i-2})$. The sequences $(\tilde U_{2i-1})_{i \geq 1}$
and $(\tilde U_{2i})_{i \geq 1}$ are sequences of martingale
differences with respect respectively to $(\F^U_{2i-1})$ and
$(\F^U_{2i})$. Substituting the variables $\tilde U_i$ to the
initial variables, in the inequality \eqref{dec1FN}, we derive
the following upper bound
\begin{multline} \label{dec15FN}
\max_{1 \leq k \leq  n } |S_k |\leq  2q M +
 \max_{2 \leq 2j \leq [n/q]} \Bigl | \sum_{i=1}^j \tilde U_{2i} \Bigr | +
\max_{1 \leq 2j-1 \leq [n/q]} \Bigl | \sum_{i=1}^j \tilde U_{2i
-1} \Bigr |
  + \sum_{i=1}^{ [n/q]} |  U'_{i} -\tilde U_{i} | +  \sum_{k=1}^n
| X''_k | \, .
\end{multline}
Since $\sum_{\ell=1}^L |a_{\ell}| \leq 1$, $|U'_i| \leq 2qM$
almost surely. Consequently $ |\tilde U_i | \leq 4qM$ almost
surely. Applying Proposition \ref{pinelis} of the appendix with
$y = 2s_n^2$, we derive that
  \begin{equation} \label{dec5}
  \begin{split}
  \p \Big (  \max_{2 \leq 2j \leq [n/q]} \Bigl | \sum_{i=1}^j \tilde
  U_{2i} \Bigr | \geq  x  \Big
  ) & \leq  2 \exp \left ( - \frac{s_n^2}{8(qM)^2} h \left( \frac{ 2 x q M }{  s_n^2}\right) \right ) \\
  & \ \ \ + \p \Big ( \sum_{i=1}^{[[n/q]/2]}
  \E (  \tilde U_{2i} ^2 | \F^U_{2(i-1)} ) \geq 2 s_n^2 \Big )
  \, .
  \end{split}
  \end{equation}
Since $\E (  \tilde U_{2i} ^2 | \F^U_{2(i-1)} ) \leq \E (
(U'_{2i} )^2 | \F^U_{2(i-1)} )$,
  \begin{equation}
  \label{dec6}
  \p \Big ( \sum_{i=1}^{[[n/q]/2]}
  \E (  \tilde U_{2i} ^2 | \F^U_{2(i-1)} ) \geq 2 s_n^2\Big )
  \leq \p \Big ( \sum_{i=1}^{[[n/q]/2]}
  \E ( (  U'_{2i})^2 | \F^U_{2(i-1)} ) \geq 2 s_n^2 \Big )  \, .
  \end{equation}
By stationarity
\[
\sum_{i=1}^{[[n/q]/2]} \E( (   U'_{2i})^2) = [[n/q]/2] \E (   S'_{q}
)^2 = [[n/q]/2] \sum_{|i| \leq q } (q -|i|) \E (X_0' X_{|i|}') \, .
\]
Now,
\[
\E (X_0' X_{|i|}') = \sum_{\ell=1}^L  \sum_{k=1}^L a_{\ell} a_{k}
\Cov \bigl ( g_M \circ f_{\ell}(Y_0) , g_M \circ f_{k}(Y_{|i|})
\bigr ) .
\]
Applying Theorem 1.1 in Rio (2000) and noticing that $Q_{ |g_M
\circ f_{\ell}(Y_{|i|})|} (u) \leq Q_{ |f_{\ell}(Y_{|i|})|} (u)
\leq Q(u)$, we derive that
\begin{equation*}
\big | \Cov \big ( g_M \circ f_{\ell}(Y_0) , g_M \circ
f_{k}(Y_{|i|}) \big ) \big | \leq 2 \int_0^{2\bar \alpha (g_M \circ
f_{\ell}(Y_0), g_M \circ f_{k}(Y_{|i|}) )} Q^2 (u) du \, ,
\end{equation*}
where
\[
\bar \alpha (g_M \circ f_{\ell}(Y_0), g_M \circ
f_{k}(Y_{|i|}) ) = \sup_{(s,t) \in {\mathbb R}^2} \big | \Cov (
\I_{ g_M \circ f_{\ell}(Y_0)\leq s},\I_{ g_M \circ
f_{k}(Y_{|i|})\leq t} )\big | \, .
\]
Since $g_M \circ f_{k}$ is monotonic on an interval and zero
elsewhere, it follows that $\{ g_M \circ f_k (x) \leq t\}$ is
either some interval or the complement of some interval. Hence
\[
\bar \alpha (g_M \circ f_{\ell}(Y_0), g_M \circ f_{k}(Y_{|i|})
) \leq 2 \bar \alpha (g_M \circ f_{\ell}(Y_0), Y_{|i|} ) \leq
\alpha_1 (|i|) \, .
\]
Consequently since $\sum_{\ell=1}^L |a_{\ell}| \leq 1$, we get that
\begin{equation}\label{covinerio}
\E (X_0' X_{|i|}')  \leq 2 \int_0^{2\alpha_{1,{\bf Y}} (|i|) } Q^2 (u) du \leq
4   \int_0^{\alpha_{1,{\bf Y}}(|i|) } Q^2 (u) du\, ,
\end{equation}
so that
\[
\sum_{i=1}^{[[n/q]/2]} \E ((   U'_{2i})^2) \leq 4 n
\sum_{i=0}^{q-1} \int_0^{\alpha_{1,{\bf Y}}(i)}Q^2(u)du \leq s_n^2 \, .
\]
This bound and Markov's inequality imply that
\begin{equation}
  \label{dec61}
  \p \Big ( \sum_{i=1}^{[[n/q]/2]}
  \E ( (  U'_{2i} )^2 | \F^U_{2(i-1)} ) \geq 2 s_n^2 \Big )
  \leq \frac{1}{s_n^2} \sum_{i=1}^{[[n/q]/2]} \E |
  \E ( (  U'_{2i} )^2 | \F^U_{2(i-1)} ) -
  \E ( (  U'_{2i} )^2) | \, .
\end{equation}

Obviously similar computations allow to treat the quantity
$\max_{1 \leq 2j-1 \leq [n/q]} | \sum_{i=1}^j \tilde U_{2i -1}|
$. Hence we get that
 \begin{multline*}
 \p \Big (  \max_{2 \leq 2j \leq [n/q]} \Bigl| \sum_{i=1}^j
\tilde U_{2i} \Bigr| +\max_{1 \leq 2j-1 \leq [n/q]} \Bigl|
\sum_{i=1}^j \tilde U_{2i -1} \Bigr| \geq  2x  \Big
 ) \leq   4 \exp \left ( - \frac{s_n^2}{8(qM)^2} h \left( \frac{ 2 x q M }{  s_n^2}\right) \right
 ) \\
 + \frac{1}{ s_n^2} \sum_{i=1}^{[n/q]} \E |
  \E ( ( U'_{i} )^2 | \M_{(i-2)q} ) -
   \E  ((  U'_{i} )^2 )|
 \, .
\end{multline*}
By stationarity we have
  \begin{equation}
  \label{wioumgvwmxvj}
  \begin{split}
  \sum_{i=1}^{[n/q]} \|
  \E ( ( U'_{i} )^2 | \M_{(i-2)q} ) -
  \E  ((  U'_{i} )^2) \|_1
  &\leq \frac{n}{q}\Vert \E ( ( S_q'  )^2 | \M_{-q}
  ) -
  \E  (( S_q'  )^2)\Vert_1 \\
  &\leq   \frac{n}{q} \sum_{i=q+1}^{2q}
  \sum_{j=q+1}^{2q}\Vert \E (  X_i'  X_j'  | \M_{0} ) -
  \E  (X_i'  X_j' )\Vert_1
  \, .
  \end{split}
  \end{equation}
Let us now prove that
  \begin{equation}
  \label{jqlskmdfj}
  \Vert \E (  X_i'  X_j'  | \M_{0} ) -
   \E  (X_i'  X_j' )\Vert_1 \leq 16 M^2 \alpha_{2, {\bf Y}}(q).
  \end{equation}
Setting $A:=\sign \{ \E (  X_i'  X_j'  | \M_{0} ) - \E  ( X_i'
X_j' )\}$, we have that
  \begin{multline*}
  \Vert \E (  X_i'  X_j'  | \M_{0} ) -
  \E  (X_i'  X_j' )\Vert_1
  =
  \E \Big \{ A \Big ( \E (  X_i'  X_j'  | \M_{0} ) -
  \E  (X_i'  X_j' ) \Big ) \Big \}   = \E  \big ( (A -\E A) X_i'  X_j' \big ) \\
  = \sum_{\ell=1}^L  \sum_{k=1}^L a_{\ell} a_{k} \E  \big ( (A
  -\E A) (g_M \circ f_{\ell}(Y_i) - \E g_M \circ
  f_{\ell}(Y_i))(g_M \circ f_{k}(Y_j) - \E g_M \circ
  f_{k}(Y_j))\big )
  \, .
  \end{multline*}
From Proposition 6.1 and Lemma 6.1 in Dedecker and Rio (2008), noticing that $Q_A (u)
\leq 1$ and $Q_{|g_M \circ f_{\ell}(Y_i)|} (u) \leq M$, we have that
\begin{multline*}
\vert \E  \big ( (A -\E A) (g_M \circ f_{\ell}(Y_i) - \E g_M
\circ f_{\ell}(Y_i))(g_M \circ f_{k}(Y_j) - \E g_M \circ
f_{k}(Y_j))\big )
\vert \\
   \leq 8 M^2 \bar \alpha ( A, g_M \circ f_{\ell}(Y_i), g_M \circ f_{k}(Y_j) )
 \, ,
\end{multline*}
where for real valued random variables $A,B,V$,
\[
  \bar \alpha (A, B,V) = \sup_{(s,t,u) \in {\mathbf R}^3} \big | \E ( (\I_{ A
  \leq s} - \p (A \leq s)) (\I_{ B \leq t} - \p (B \leq t)) (\I_{ V
  \leq u} - \p (V \leq u)))\big | \, .
\]
For all $i,j \geq q$,
\[
\bar \alpha ( A, g_M \circ f_{\ell}(Y_i), g_M \circ f_{k}(Y_j)
) \leq 4 \bar \alpha ( A, Y_i, Y_j )\leq 2 \alpha_{2, {\bf Y}}
(q) \, .
\]
This concludes the proof of \eqref{jqlskmdfj}. Together with
\eqref{wioumgvwmxvj}, this yields
\begin{equation} \label{p18MDA}
\sum_{i=1}^{[n/q]} \E |
  \E ( ( U'_{i} )^2 | \M_{(i-2)q} ) -
   \E  (  U'_{i} )^2 |  \leq  16 nq M^2 \alpha_{2, {\bf Y}}(q)
  \, .
\end{equation}
It follows that
\begin{multline} \label{decinter}
\p \Big (  \max_{2 \leq 2j \leq [n/q]} \Bigl | \sum_{i=1}^j
\tilde U_{2i} \Big | +\max_{1 \leq 2j-1 \leq [n/q]} \Big |
\sum_{i=1}^j \tilde U_{2i -1} \Big | \geq  2x  \Big
 ) \\
  \leq   4 \exp \left ( -
\frac{s_n^2}{8(qM)^2} h \left( \frac{ 2 x q M }{  s_n^2}\right)
\right
 )   + \frac{16 n qM}{ s_n^2} M
\alpha_{2, {\bf Y}}(q)
 \, .
\end{multline}
Now by using Markov's inequality, we get that
\[
 \p \Big ( \sum_{i=1}^{ [n/q]} |  U'_{i} -\tilde U_{i} | +
\sum_{k=1}^n  | X_k'' | \geq  x  \Big
 ) \\
    \leq \frac{1}{x} \Big ( \sum_{i=1}^{[n/q]} \| \E (  U'_{i}  | {\cal
M}_{(i-2)q} )\|_1 +  \sum_{k=1}^n \| X''_k \|_1  \Big )\, .
\]
By stationarity, we have that
\begin{eqnarray*}
 \sum_{i=1}^{[n/q]} \|
\E (  U'_{i}  | {\cal M}_{(i-2)q} )\|_1 \leq \frac{n}{q}
\sum_{i=q+1}^{2q}\| \E ( X'_i | {\cal M}_{0} )\|_1 \, .
\end{eqnarray*}
Setting $A = \sign \{\E ( X'_i | {\cal M}_{0} )\}$, we get that
\[
\| \E ( X'_i | {\cal M}_{0} )\|_1 = \E ((A - \E A)X_i')=
\sum_{\ell=1}^L a_{\ell} \E \big ( (A - \E A) (g_M \circ
f_{\ell}(Y_i) - \E g_M \circ f_{\ell}(Y_i))\big )
\]
 Now  applying again Theorem 1.1 in  Rio (2000),
 and using the fact that $Q_{|g_M \circ f_{\ell}(Y_i)|}
(u) \leq Q(u)$, we derive that
\begin{eqnarray*}
& & \vert  \E \big ( (A - \E A) (g_M \circ f_{\ell}(Y_i) - \E g_M
\circ f_{\ell}(Y_i))\big )
  \leq 2 \int_0^{ 2\bar \alpha ( A, g_M \circ
f_{\ell}(Y_i))} Q(u) du
 \, .
\end{eqnarray*}
Since for all $i \geq q$,
\[
\bar \alpha ( A, g_M \circ f_{\ell}(Y_i) ) \leq 2 \bar
\alpha(A, Y_i) \leq  \alpha_{1, {\bf Y}} (i) \leq  \alpha_{2,
{\bf Y}} (i) \, ,
\]
we derive that
  \begin{equation}
  \label{majnorm1prime}
  \| \E ( X'_i | {\cal M}_{0} )\|_1 \leq 4 \int_0^{\alpha_{2, {\bf Y}} (i)} Q(u) du \, ,
  \end{equation}
which implies that
  \begin{equation}
  \label{decinter2}
  \p\Big ( \sum_{i=1}^{ [n/q]} |  U'_{i} -\tilde U_{i} | +
  \sum_{k=1}^n  | X_k'' | \geq  x  \Big
  ) \leq \frac{ 4 n}{x}\int_0^{\alpha_{2, {\bf Y}} (q)} Q (u) du
  + \frac{1}{x} \sum_{k=1}^n \E ( |
  X_k''|) \, .
  \end{equation}
Then starting from \eqref{dec15FN}, if $q$
and $M$ are chosen in such a way that $qM \leq x$, we derive from
\eqref{decinter} and \eqref{decinter2} that
 \begin{equation}
 \label{decinter3}
 \begin{split}
  \p \Big ( \max_{1 \leq k \leq  n } |S_k | \geq 5x \Big) & \leq
    4 \exp \left ( - \frac{s_n^2}{8(qM)^2} h \left( \frac{ 2 x q
  M }{  s_n^2}\right) \right
  ) + \frac{16 n qM}{ s_n^2} M
  \alpha_{2, {\bf Y}}(q) \\
  & \ \ \     + \frac{4n}{x}\int_0^{\alpha_{2, {\bf Y}} (q)} Q (u) du
  +  \frac{1}{x} \sum_{k=1}^n \E ( |X_k''|)
  \, .
  \end{split}
  \end{equation}
Now choose $v = S(x/r)$, $q = \min\{q\in \N \tq \alpha_{2,{\bf
Y}}(q)\leq v\} \wedge n $ and $M = Q (v)$. Since $R$ is right
continuous, we have $R(S(w))\leq w$ for any $w$, hence
\[
qM = R(v) = R (S(x/r)) \leq x/r \leq x \, .
\]
Note also that, writing $\varphi_M(x)=(|x|-M)_+$,
\[
\sum_{k=1}^n \E ( | X_k''|) \leq 2 \sum_{\ell =1}^L
|a_{\ell}|\sum_{k=1}^n \E (\varphi_{M} ( f_{\ell} (Y_k)))
\]
and that $Q_{\varphi_M ( f_{\ell} (Y_k))} \leq Q_{| f_{\ell}
(Y_k) |}\I_{[0,v]} \leq Q\I_{[0,v]} $. Consequently
  \begin{equation} \label{dec13FN}
  \sum_{k=1}^n  \E ( | X_k''|)
  \leq  2 \sum_{\ell =1}^L |a_{\ell}|\sum_{k=1}^n  \int_0^v  Q_{|
  f_{\ell} (Y_k) |} (u) du
  \leq  2n \int_0^v Q(u) du \, .
  \end{equation}

Assume first $q<n$. The choice of $q$ then implies that
$\alpha_{2, {\bf Y}}(q) \leq v $ and $M \alpha_{2, {\bf Y}} (q)
\leq v Q(v)\leq \int_0^v Q(u)du$. Moreover, as $qM\leq x/r$, we
have
  \begin{equation*}
  \frac{1}{(qM)^2} h \left( \frac{ 2 x q
  M }{  s_n^2}\right)
  \geq \frac{r^2}{x^2}h \left( \frac{ 2 x^2
  }{ r s_n^2}\right),
  \end{equation*}
since the function $t\mapsto t^{-2}h(t)$ is decreasing.
Together with \eqref{decinter3} and \eqref{dec13FN}, this gives
the desired inequality \eqref{eqmainmax}.

If $q=n$, the previous argument breaks down since we may have
$\alpha_{2, {\bf Y}}(q) > v $. However, a much simpler argument
is available. Indeed, bounding simply $X'_i$ by $2M$, we obtain
$\max_{1\leq k\leq n}|S_k| \leq 2qM + \sum_{k=1}^n |X''_k|$.
Since $2qM\leq 2x$, this gives
  \begin{equation*}
  \p \Big ( \max_{1 \leq k \leq  n } |S_k | \geq 5x \Big)
  \leq \frac{1}{x} \sum_{k=1}^n  \E ( | X_k''|).
  \end{equation*}
With \eqref{dec13FN}, this again implies \eqref{eqmainmax}.

The proposition is proved for any variable $X_i=f(Y_i) - \E(f(Y_i)) $ with
$f=\sum_{\ell=1}^L a_{\ell} f_{\ell}$ and $f_{\ell} \in \tMon (Q, P_{Y_0})$,
$\sum|a_\ell|\leq 1$. Since these functions are dense in $\tMonm (Q, P_{Y_0})$ by
definition, the result follows by applying Fatou's lemma. \qed

\subsection{Proof of Theorem \ref{ASthm}}

Let us first prove the  inequality (\ref{r1app2LIL}). We follow
the proof of Theorem 6.4 page 89 in Rio (2000), and we use the
same notations: $Lx=\ln(x \vee e)$ and $LLx=\ln( \ln (x \vee e)
\vee e)$. Let $A$ be as in (\ref{r1app2LIL}). We apply
Proposition \ref{FN} with
\[
r=r_n=8LLn, \quad
 x=x_n=(A\sqrt{2n LLn})/5 \quad \text{and} \quad  s_n=x_n/\sqrt{r_n} \, .
\]
We obtain
\[
\sum_{  n >0}  \frac 1n {\mathbb P} \Big ( \sup_{1 \leq k \leq n} |S_k |
\geq A \sqrt{2 n LLn} \Big ) \leq 4 \sum_{n>0} \frac{1}{n 3^{LLn}}+
22\sum_{n>0}\frac{1}{x_n}\int_0^{S(x_n/r_n)} Q(u) du \, .
\]
Clearly the first series on right hand converges.  From the end
of the proof of Theorem 6.4 in Rio (2000), we see that the
second  series on the right hand side converges. This completes
the proof of (\ref{r1app2LIL}).

Note that the inequality \eqref{r1app2LIL} implies that
\begin{eqnarray} \label{r2app2LIL}
  \limsup_{  n \ra \infty}  \frac{|
S_n |}{\sqrt{2 n LLn} }  \leq 20 \Big (\sum_{k \geq 0}
\int_0^{\alpha_{1, {\bf Y}}(k)} Q^2(u) du \Big )^{1/2} \mbox{
almost surely}\, .
\end{eqnarray}

We turn now to the proof of \eqref{asr1}. Assume that $\theta$ is ${\mathbb P}$-ergodic. In 1973, Gordin (see
also Esseen and Janson (1985)) proved that if
\begin{equation}
\label{condgord1} \sum_{k \geq 1} \| {\mathbb E} (X_k
|{\mathcal M}_0) \|_1 < \infty
\end{equation} and \begin{equation} \label{condgord2} \liminf_{n \rightarrow \infty}
\frac{1}{\sqrt n} {\mathbb E} \Big( \Big |\sum_{k=1}^n X_k \Big
|\Big) < \infty \, ,
\end{equation} then
$ X_0 = D_0 + Z_0-Z_0 \circ \theta \, , $ where
$\|Z_0\|_1<\infty$, ${\mathbb E}(D_0^2)<\infty$, $D_0 $ is
${\mathcal M}_0$-measurable, and ${\mathbb E} (D_0 |{\mathcal
M}_{-1})=0$.

Notice now that by a similar computation than to get
\eqref{majnorm1prime}, we have that
\begin{equation} \label{majgamma}
\| {\mathbb E} (X_k |{\mathcal F}_0) \|_1 \leq 4 \int_0^{\alpha_{1, {\bf Y}}
(k)} Q(u) du \, .
\end{equation}
Hence \eqref{condQ} implies \eqref{condgord1}. Now clearly
\eqref{condgord2} holds as soon as $\sum_{k=0}^{\infty} |\Cov
(X_0,X_k)| < \infty$ which holds under \eqref{condQ} by
applying the upper bound \eqref{covinerio} with $M=\infty$
(note that this also justifies the convergence of the series
$\sigma^2$).

Consequently, if we set $D_i=D_0 \circ \theta^i$, and $Z_i=Z_0
\circ \theta^i$, we then obtain under \eqref{condQ} that
  \begin{equation}
  \label{coboundSm} S_{n}=M_{n} + Z_1-Z_{n+1} \, ,
  \end{equation}
where $M_{n}=\sum_{j=1}^nD_j$ is a martingale in ${\mathbb L}^2$ and $Z_0$ is integrable.
Now \eqref{asr1} follows by the almost sure invariance principle for martingales (see
Theorem 3.1 in Berger (1990)) if we can prove that
  \begin{equation} \label{negRnas} Z_n =o(\sqrt {n
  LLn})\, ,
  \quad  \text{almost surely.}
  \end{equation}
According to the lemma page 428 in Voln\'{y} and Samek (2000), we
have either \eqref{negRnas} or
\begin{equation}
\label{notneg} {\mathbb P} \Big( \limsup_{n \rightarrow \infty}
\frac{|Z_n|}{\sqrt {n LLn}} = \infty \Big )=1\, .
\end{equation}
Using the decomposition \eqref{coboundSm}, the fact that $M_n$
satisfies the law of the iterated logarithm and that $S_n$
satisfies \eqref{r2app2LIL}, it is clear that \eqref{notneg}
cannot hold, which then proves \eqref{negRnas} and ends the
proof of \eqref{asr1}.
\qed

\section{Proofs of the dynamical estimates}
\label{secdyn}

If $f$ is supported in $[0,1]$, let $V(f)$ be the variation of
the function $f$, given by
  \begin{equation*}
  \Var(f)=\sup_{x_0<\dots<x_N} \sum_{i=1}^N | f(x_{i+1})-f(x_i)|\, ,
  \end{equation*}
where the $x_i$s are real numbers (not necessarily in $[0,1]$).
Note that $\Var(.)$ is a norm and that $\Var(f\cdot g)\leq
\Var(f)\Var(g)$.

Let us fix once and for all a GPM map $T:[0,1]\to [0,1]$ of
parameter $\gamma\in(0,1)$. Let $v_k: T_{(k)} I_k \to I_k$ be
the inverse branches of $T$. Consider $M=\{m\in \{1,\dots,d-1\}
\tq 0 \in T_{(m)}I_m\}$, and let $z_0\in (0,y_1)$ be so small
that $v_m$ is well defined on $[0,z_0]$ for any $m\in M$,
$v_0'$ is decreasing on $(0,z_0]$ (this is possible since
$v_0''(x)<0$ for small $x$), and $T_{(k)}I_k\cap
[0,z_0]=\emptyset$ for $k\not\in M$. Note that $M \neq \emptyset$, since $T$
is topologically transitive.

Define a sequence $z_n$ inductively by $z_n=v_0(z_{n-1})$. Let
$J_n=(z_{n+1},z_n]$, so that $T^n$ is bijective from $J_n$ to
$(z_1,z_0]$. Following the procedure in Zweim\"{u}ller (1998), the
invariant measure of $T$ may be constructed as follows: we
first consider the first return map on $(z_1,1]$. It is Rychlik
and topologically transitive, hence it admits an invariant
measure $\nu_0$ on $(z_1,1]$ whose density $h_0$ is bounded from above and
below in $(z_1, 1]$ and has bounded variation. Extending $\nu_0$ to the whole
interval by the formula
$$
\nu(A)=\nu_0(A \cap (z_1, 1]) + \sum_{n \geq 1} \nu_0(T^{-n}(A) \cap \{ \phi > n \})\, ,
$$
where $\phi$ is the first return time to
$(z_1,1]$, and then renormalizing, we obtain the invariant
probability measure of $T$. Denoting by $h$ the density of
$\nu$, the previous formula becomes, for $x\in [0,z_1]$,
  \begin{equation}
  \label{formule_pour_h}
  h(x)=\sum_{n=0}^\infty \sum_{m\in M} |(v_m v_0^n)'(x)| h(v_m v_0^n x).
  \end{equation}

Our goal in this paragraph and the next is to study the
Perron-Frobenius operator $ K^n$ acting on the space $\BV$ of
bounded variation functions. Let $ K(x,y)$ be the kernel
corresponding to the operator $ K$. It is given by $ K(x, v_k
x)= h(v_k x) |v'_k(x)|/h(x)$ for $k\in\{0,\dots,d-1\}$,
and $ K(x,y)=0$ if $y$ is not of the form $v_k x$. By
definition,
  \begin{equation*}
  K^n f(x_0)=\sum_{x_1,\dots,x_n}  K(x_0,x_1) K(x_1,x_2)
  \dots  K(x_{n-1},x_n) f(x_n)\, .
  \end{equation*}
To understand the behavior of $ K^n$, we will break the
trajectories $x_0,\dots, x_n$ of the random walk according to
their first and last entrance in the reference set $(z_1,1]$ --
the interest of this set is that $T$ is uniformly expanding
there. More precisely, let us define operators $A_n,B_n,C_n$
and $T_n$ as follows: they are defined like $ K^n$ but we only
sum over trajectories $x_0,\dots,x_n$ such that
\begin{itemize}
\item For $A_n$, $x_0,\dots,x_{n-1} \in [0,z_1]$ and $x_n\in
(z_1,1]$.
\item For $B_n$, $x_0\in (z_1,1)$ and $x_1,\dots,x_n\in [0,z_1]$.
\item For $C_n$, $x_0,\dots,x_n\in [0,z_1]$.
\item For $T_n$, $x_0\in (z_1,1]$ and $x_n\in (z_1,1]$.
\end{itemize}
By construction, one has the decomposition
  \begin{equation}
  \label{eqdec}
   K^n f= \sum_{a+k+b=n} A_a T_k B_b f + C_n f \, .
  \end{equation}
One can give formulas for $A_n,B_n$ and $C_n$, as follows:
\begin{eqnarray}\label{opAn}
A_n f(x)&=&\I_{[0, z_1]}(x) \sum_{m\in M} \frac{|(v_m v_0^{n-1})'(x)|  h(v_m v_0^{n-1} x)}{ h(x)} f(v_mv_0^{n-1} x)\, ,\\
\label{opBn}
B_n f(x)&=& \I_{(z_1,z_0]}(x)\frac{(v_0^n)'(x)  h(v_0^n x)}{ h(x)}f(v_0^n x)\, ,\\
\label{opCn}
C_n f(x)&=&\I_{[0,z_1]}(x)\frac{(v_0^n)'(x) h(v_0^n x)}{ h(x)} f(v_0^n x)\, .
\end{eqnarray}

On the other hand, the operator $T_n$ is less explicit, but it
can be studied using operator renewal theory.
\begin{prop}
The operator $T_n$ can be decomposed as
  \begin{equation}\label{Gouezel07}
  T_n f= \left(\int_{(z_1,1]} f d\nu\right)\I_{(z_1,1]} + E_n f\ ,
  \end{equation}
where the operator $E_n$ satisfies $\displaystyle \Var(E_n f)\leq
\frac{C}{n^{(1-\gamma)/\gamma}}\Var(f)$.
\end{prop}
\begin{proof}
Since this follows closely from the arguments in Sarig (2002),
Gou\"{e}zel (2004b) and Gou\"{e}zel (2007), we will only sketch the
proof.

Define an operator $R_n$ by $R_n f(x_0)=\I_{(z_1,1]}(x)\sum
K(x_0,x_1)\dots K(x_1,x_n) f(x_n)$, where the summation is over
all $x_1,\dots,x_{n-1} \in [0,z_1]$ and $x_n \in (z_1,1]$: this
operator is similar to $T_n$, but it only takes the first
returns to $(z_1,1]$ into account. Breaking a trajectory into
its successive excursions outside of $(z_1,1]$, it follows that
the following renewal equation holds: $T_n=\sum_{\ell=1}^\infty
\sum_{k_1+\dots+k_\ell=n}R_{k_1}\dots R_{k_\ell}$. In other
words, $I+\sum T_n z^n=(I-\sum R_k z^k)^{-1}$, at least as
formal series.

In the proof of Lemma 3.1 in Gou\"{e}zel (2007), it is shown that
the operators $R_k$ act continuously on $\BV$, with a norm
bounded by $C/k^{1+1/\gamma}$ -- the estimates in Gou\"{e}zel do
not deal with the factor $h$, but since this function as well
as its inverse have bounded variation on $(z_1,1]$ they do not
change anything. Since this is summable, we can define, for
$|z|\leq 1$, an operator $R(z)=\sum R_n z^n$ acting on $\BV$.
Moreover, Gou\"{e}zel (2007) also proves that the essential
spectral radius of this operator is $<1$ for any $|z|\leq 1$.
Thanks to the topological transitivity of $T$, it follows that
$R(1)$ has a simple eigenvalue at $1$ (the corresponding
eigenfunction is the constant function $1$), while $I-R(z)$ is
invertible for $z\not=1$.

This spectral control makes it possible to apply Theorem 1.1 in
Gou\"{e}zel (2004b), dealing with renewal sequences of operators as
above. Its conclusion implies \eqref{Gouezel07}.
\end{proof}

With \eqref{eqdec}, we finally  obtain that
  \begin{equation}
  \label{eq_somme}
  K^n f = \sum_{a+k+b=n} A_a (\I_{(z_1, 1]}) \cdot \nu( B_b f) +
  \sum_{a+k+b=n} A_a E_k B_b f + C_n f \, ,
  \end{equation}
where
\begin{equation}
  \label{eqfinalEk}
  \Var(E_k f)\leq \frac{C} {k^{(1-\gamma)/\gamma}} \Var(f).
  \end{equation}

\subsection{Proof of Proposition \ref{propvar}}
\label{subsecproof}

We shall prove successively that, for $n>0$,
\begin{eqnarray}
  \label{eqfinalCn}
  \Var(C_n f)&\leq& C \Var(f)\, , \\
  \label{eqfinalAa}
  \Var(A_n f) &\leq& C \Var(f)/(n+1)\, , \\
  \label{eqfinalBb}
  \Var(B_n f) &\leq& C \Var(f)/(n+1)^{1/\gamma}\, .
\end{eqnarray}
The proof of Proposition \ref{propvar} follows from the above
upper bounds and from the following elementary lemma.
\begin{lma}
\label{lem_convole} Let  $u_n$ and  $v_n$ be two non increasing sequences such that
$u_{[n/2]} \leq Cu_n$ and $v_{[n/2]}\leq C v_n$. Then
  \begin{equation*}
  \sum_{i+j=n} u_i v_j \leq C u_n \left(\sum_{j=0}^n v_i\right)
  + C v_n \left(\sum_{i=0}^n u_i \right)\, .
  \end{equation*}
\end{lma}
\begin{proof}
If  $i\leq n/2$, we use that  $v_j$ is bounded by $Cv_n$. If
$j\leq n/2$, we use that  $u_i$ is bounded by  $Cu_n$.
\end{proof}

\medskip

We can now complete the proof, assuming the bounds
\eqref{eqfinalCn}, \eqref{eqfinalAa}, and \eqref{eqfinalBb}:
\begin{proof}[Proof of Proposition \ref{propvar}]
Let $f$ be such that $\nu(f)=0$. We will bound $V(K^n f)$ using
the decomposition of $K^n f$ given in \eqref{eq_somme}. Using
\eqref{eqfinalAa}, \eqref{eqfinalEk} and \eqref{eqfinalBb}, we
get
  \begin{equation*}
  \Var\left(\sum_{a+k+b=n} A_a E_k B_b f \right)
  \leq C \Var(f) \sum_{a+k+b=n} \frac{1}{(a+1)(k+1)^{(1-\gamma)/\gamma}(b+1)^{1/\gamma}}\, .
  \end{equation*}
By lemma \ref{lem_convole},
\[
 \sum_{k+b=j}\frac{1}{(k+1)^{(1-\gamma)/\gamma}(b+1)^{1/\gamma}}
 \leq
 \frac{C}{(j+1)^{(1-\gamma)/\gamma}}
 \]
and \[ \sum_{a+j=n}\frac{1}{(a+1) (j+1)^{(1-\gamma)/\gamma}}
\leq C\Big ( \frac{\ln (n)}{
 (n+1)^{(1-\gamma)/\gamma}} \vee \frac 1n \Big)\, .
\]
Consequently,
  \begin{equation}\label{premierterm}
  \Var\left(\sum_{a+k+b=n} A_a E_k B_b f \right)
  \leq C \Var(f)\Big ( \frac{\ln (n)}{
 (n+1)^{(1-\gamma)/\gamma}} \vee \frac 1n \Big)\, .
  \end{equation}
It remains to bound up the first term in  \eqref{eq_somme},
which can be written
  \begin{equation*}
  \sum_{a=0}^n A_a (\I_{(z_1,1]}) \cdot \left(\sum_{b=0}^{n-a} \nu( B_b f)\right)\, .
  \end{equation*}
Now, $\sum_{b=0}^\infty  \nu(B_b f) = \nu(f)=0$, so that
  \begin{equation*}\label{controlcentre}
  \left|\sum_{b=0}^{n-a} \nu( B_b f ) \right| =
   \left| \sum_{b>n-a} \nu( B_b f ) \right|
  \leq \sum_{b>n-a} \Var(B_b f)
  \leq \sum_{b>n-a}  \frac{C\Var(f)}{(b+1)^{1/\gamma}}
  \leq  \frac{D\Var(f)}{(n+1-a)^{(1-\gamma)/\gamma}}\, .
  \end{equation*}
By \eqref{eqfinalAa}, $\Var(A_a \I_{(z_1,1]}) \leq C/(a+1)$.
Consequently,
  \begin{equation}
  \label{secondterm}
  \begin{split}
  \Var\left(\sum_{a=0}^n A_a (\I_{(z_1,1]}) \cdot \left(\sum_{b=0}^{n-a} \nu( B_b f)\right)\right)
  &\leq
  C \Var(f) \sum_{a=0}^n \frac{1}{(a+1)(n+1-a)^{(1-\gamma)/\gamma}}
  \\&
  \leq D \Var(f)\Big ( \frac{\ln (n)}{
 (n+1)^{(1-\gamma)/\gamma}} \vee \frac 1n \Big)\, ,
  \end{split}
  \end{equation}
the last inequality following from Lemma \ref{lem_convole}.

Starting from \eqref{eq_somme} and using \eqref{eqfinalCn},
\eqref{premierterm} and \eqref{secondterm} we obtain that
$\Var(K^n f)\leq C\Var(f)$ for any $f$ such that $\nu(f)=0$.
Now let $f$ be any $\BV$ function on $[0, 1]$, and let $\|df \|$ be the variation norm of the measure $df$ on $[0, 1]$. To conclude the proof,
it suffices to note that $\|dK^n(f)\|= \|dK^n(f^{(0)})\| \leq \Var(K^n(f^{(0)})) \leq C\Var(f^{(0)}) \leq 3C\|df\|$.
\end{proof}

It remains to prove the upper bounds \eqref{eqfinalCn},
\eqref{eqfinalAa}, and \eqref{eqfinalBb}. We shall use the
following facts, proved e.g.~in Liverani, Saussol and Vaienti
(1999) or Young (1999). We will denote Lebesgue measure by
$\lambda$.
\begin{enumerate}
\item One has $z_n\sim C /n^{1/\gamma}$ for some $C>0$. Moreover,
$\lambda(J_n)=z_n-z_{n+1} \sim C/n^{(1+\gamma)/\gamma}$ for
some $C>0$. One has
  \begin{equation}
  \label{hxn}
  h(z_n)\sim C z_n^{-\gamma} \sim D n\, .
  \end{equation}
\item There exists a constant $C>0$ such that, for all $n\geq
0$ and $k\geq 0$, and for all $x,y\in J_{k}$,
  \begin{equation*}
  \left| 1 -\frac{ (v_0^n)'(x)}{(v_0^n)'(y)}\right|
  \leq C |x-y|\, .
  \end{equation*}
Integrating the above inequality, we obtain that
  \begin{equation}
  \label{bornev0n}
  C^{-1} \frac{ \lambda(J_{n+k})}{\lambda(J_{k})} \leq (v_0^n)'(x) \leq C
  \frac{ \lambda(J_{n+k})}{\lambda(J_{k})}\, .
  \end{equation}
\item The function $(v_0^n)'$ is decreasing on $[0,z_1)$.
\end{enumerate}

The following easy lemma follows from the definition of $\Var$.
\begin{lma}
\label{lemmonotone} If $f$ is nonnegative and monotonic on some interval $I$, then
  \begin{equation}
  \label{mon}
  \Var(\I_I f) \leq C \sup_I |f|.
  \end{equation}
If $f$ is positive on some interval $I$, then
  \begin{equation}
  \label{var_inverse}
  \Var(\I_I/f) \leq C \Var(\I_I f)/\min_I |f|^2.
  \end{equation}
\end{lma}

We shall also use the following lemma on the density $h$.
\begin{lma}
\label{lem_densite_h}
There exists a constant $C$ such that, for any $1\leq i < j$,
  \begin{equation}
  \Var( \I_{[z_j,z_i]} h) \leq C j \quad \text{ and } \Var(\I_{[z_j,z_i]} /h)\leq Cj/i^2.
  \end{equation}
\end{lma}
\begin{proof}
We start from the formula \eqref{formule_pour_h} for $h$, and
the inequality $\Var(fg)\leq \Var(f)\Var(g)$, to obtain
  \begin{equation}
  \Var( \I_{[z_j,z_i]} h) \leq \sum_{n=0}^\infty \sum_{m\in M}
  \Var( \I_{[z_j,z_i]}(v_0^n)') \cdot \Var(\I_{[z_j,z_i]}|v_m'\circ v_0^n|)\cdot \Var(\I_{[z_j,z_i]} h\circ  v_m \circ v_0^n)
  \end{equation}
Since the functions $v_m'$ have bounded variation, and the
function $h$ has bounded variation on $(z_1,1]$ (which contains
the image of $v_m v_0^n( 0,z_1)$), we get $\Var( \I_{[z_j,z_i]}
h) \leq C \sum_{n=0}^\infty \Var( \I_{[z_j,z_i]}(v_0^n)')$.
Since the function $(v_0^n)'$ is decreasing on $[z_j,z_i]$, we
get by using \eqref{mon}
  \begin{equation*}
  \Var( \I_{[z_j,z_i]} h) \leq C \sum_{n=0}^\infty (v_0^n)'(z_j)
  \leq C \sum_{n=0}^\infty \frac{\lambda(J_{n+j})}{\lambda(J_j)}
  = C\frac{z_j}{z_j-z_{j+1}} \leq C \frac{j^{-1/\gamma}}{j^{-1/\gamma-1}}=Cj.
  \end{equation*}
This proves the first inequality of the proposition.

To prove the second one, we use \eqref{var_inverse}. Since
$\min_{[z_j,z_i]}|h|\geq C z_i^{-\gamma} \geq C i$, the result
follows.
\end{proof}

\medskip

We can now prove the upper bounds \eqref{eqfinalCn},
\eqref{eqfinalAa}, and \eqref{eqfinalBb}

Since $C_n$ is given by \eqref{opCn}, the upper bound
\eqref{eqfinalCn} follows from Lemma \ref{lem_controleCn}
below.
\begin{lma}
\label{lem_controleCn} There exists $C>0$ such that, for any
$n\geq 1$,
  \begin{equation}
  \label{eqcoupevar}
  \Var \left(\I_{[0,z_1]}\frac{(v_0^n)'(x)h(v_0^n x)}{h(x)}\right) \leq C\, .
  \end{equation}
\end{lma}
\begin{proof}
Since $K 1=1$, we have $h(x)=v_0'(x)h(v_0 x)+\sum_{m\in M}
|v_m'(x)| h(v_m x)$ on $[0, z_1]$. By iterating this equality, we
obtain for any $n\in \N$,
  \begin{equation*}
  h(x)=(v_0^n)'(x) h(v_0^nx)+\sum_{j=0}^{n-1}\sum_{m\in M} |(v_m v_0^j)'(x)| h( v_m v_0^j x)\, .
  \end{equation*}
Consequently,
  \begin{equation}
  \label{eqastuce}
  1-\frac{(v_0^n)'(x)h(v_0^n x)}{h(x)}=\sum_{j=0}^{n-1}\sum_{m\in M} \frac{(v_0^j)'(x) |v_m'(v_0^j x)|
  h( v_m v_0^j x)}{h(x)}\, .
  \end{equation}

Let $s$ be such that $2^s \leq n <2^{s+1}$. To prove
\eqref{eqcoupevar}, we will control, for any $k$,
  \begin{equation*}
  \Var\left(\I_{[z_{2^k}, z_{2^{k-1}}]}\frac{(v_0^n)'(x)h(v_0^n x)}{h(x)}\right)\, .
  \end{equation*}
Assume first that $k\leq s$. On $[z_{2^k}, z_{2^{k-1}}]$, the
function $(v_0^n)'$ is decreasing, so that its variation is
bounded in terms of its supremum $(v_0^n)'(z_{2^{k}})\leq C
\lambda(J_{2^{k}+n})/\lambda(J_{2^{k}})$. The variation of the
function $h\circ v_0^n$ on $[z_{2^k}, z_{2^{k-1}}]$ is the
variation of $h$ on $[z_{2^k+n}, z_{2^{k-1}+n}]$, hence by
Lemma \ref{lem_densite_h} it is bounded by $C (2^k+n)$. This
lemma also shows that the variation of $1/h$ is bounded by
$C/2^k$. Hence,
  \begin{align*}
  \Var\left(\I_{[z_{2^k}, z_{2^{k-1}}]}\frac{(v_0^n)'(x)h(v_0^n x)}{h(x)}\right)&
  \leq C \frac{ \lambda(J_{2^{k}+n})}{\lambda(J_{2^{k}})} \frac{ 2^k+n}{2^k}
  \\&
  \leq C \frac{ (2^{k}+n)^{-(1+\gamma)/\gamma}}{(2^{k})^{-(1+\gamma)/\gamma}} \frac{ 2^k+n}{2^k}
  \leq C \frac{(2^k)^{1/\gamma}}{n^{1/\gamma}}\, .
  \end{align*}
Summing on $k$, we get
\begin{equation}
  \label{eqvar1bis}
  \Var\left(\I_{[z_{2^{s+1}}, z_{1}]}\frac{(v_0^n)'(x)h(v_0^n x)}{h(x)}\right)
  \leq C \sum_{k=1}^{s}\frac{(2^k)^{1/\gamma}}{n^{1/\gamma}}
  \leq \frac{C 2^{ s/ \gamma}}{n^{1/\gamma}} \leq C\, ,
  \end{equation}
since $2^s \leq n$.

Let now $k>s$. The previous upper  bound gives a suboptimal
control, hence we shall use the right hand term in
\eqref{eqastuce}. For $0\leq j\leq n-1$ and $m\in M$, the
variation of $v'_m\circ v_0^j\cdot h\circ v_m\circ v_0^j$ is
uniformly bounded (since $v_m$ is $C^2$ and $h$ has bounded
variation on $(z_1,1]$). Moreover, as above, the variation of
$(v_0^j)'$ is bounded by $C \lambda(J_{2^{k}+j})/
\lambda(J_{2^{k}})$, which is uniformly bounded. Finally, the
variation of $1/h$ is at most $C/2^k$, by Lemma
\ref{lem_densite_h}. Consequently,
  \begin{equation*}
  \Var\left(\I_{[z_{2^k}, z_{2^{k-1}}]}\left(1-\frac{(v_0^n)'(x)h(v_0^n x)}{h(x)}\right)\right)
  \leq \sum_{j=0}^{n-1} \frac{C}{2^k}
  =\frac{Cn}{2^k}\, .
  \end{equation*}
Summing on $k>s$,
  \begin{equation}
  \label{eqvar2}
  \Var\left(\I_{[0, z_{2^{s+1}}]}\left(1-\frac{(v_0^n)'(x)h(v_0^n x)}{h(x)}\right)\right)
  \leq C n\sum_{k=s+1}^\infty \frac{1}{2^k} \leq \frac{Cn}{2^s} \leq D\, .
  \end{equation}
Lemma \ref{lem_controleCn} follows by combining
\eqref{eqvar1bis} and \eqref{eqvar2}.
\end{proof}

\medskip

Since $A_n$ is given by \eqref{opAn}, the upper bound
\eqref{eqfinalAa} follows from Lemma \ref{lem_controleAn}
below.
\begin{lma}\label{lem_controleAn}
There exists a positive constant $C$ such that, for any  $n\geq
1$,
  \begin{equation}
  \label{kqlsjflm}
  \Var\left(
  \I_{[0, z_1]}(x) \sum_{m\in M} \frac{|(v_m v_0^{n-1})'(x)|  h(v_m v_0^{n-1} x)}{ h(x)}
  \right)\leq \frac{C}{n}\, .
  \end{equation}
\end{lma}
\begin{proof}
As in the proof of Lemma \ref{lem_controleCn}, we control the
variation of the functions on $[z_{2^k}, z_{2^{k-1}}]$. On this
interval, the variation of $(v_m v_0^{n-1})'$ is at most $C
\lambda(J_{2^{k}+n})/\lambda(J_{2^{k}})$, the variation of
$h(v_m v_0^{n-1})$ is bounded by $C$ and the variation of $1/h$
is bounded by $C/2^k$. Summing on $k$, we obtain
  \begin{multline*}
  \Var\left(
  \I_{[0, z_1]}(x) \sum_{m\in M} \frac{|(v_m v_0^{n-1})'(x)|  h(v_m v_0^{n-1} x)}{ h(x)}
  \right)
  \\
  \leq
  C\sum_{k=1}^\infty \frac{\lambda(J_{2^{k}+n})}{\lambda(J_{2^{k}})} \frac{1}{2^k}
  \leq D \sum_{k=1}^\infty \frac{ 2^{k(1+\gamma)/\gamma}}{(n+2^{k})^{(1+\gamma)/\gamma}}
  \frac{1}{2^k}\, .
  \end{multline*}
Let $s$ be such that  $2^s\leq n <2^{s+1}$.  We split the sum on
the sets $k\leq s$ and $k>s$, and we obtain the upper bound
  \begin{equation*}
  C\sum_{k=1}^s \frac{2^{ k(1+\gamma)/\gamma}}{(n+1)^{(1+\gamma)/\gamma}2^k}+
  C\sum_{k=s+1}^\infty \frac{1}{2^k} \\
  \leq  \frac{C 2^{s/\gamma}}{(n+1)^{(1+\gamma)/\gamma}} + \frac{1}{2^s}
  \leq \frac{D}{n}\, .
  \qedhere
  \end{equation*}
\end{proof}

It remains to prove \eqref{eqfinalBb}.   Recall that  $B_n$ is
given by \eqref{opBn}. On $(z_1, z_0]$, the variation of the
function $(v_0^n)'$ is bounded by  $C \lambda(J_n)/\lambda(J_0)
\leq C /n^{(1+\gamma)/\gamma}$, the variation of $1/h$ is
bounded by $C$, and the variation of $h(v_0^n x)$ is bounded by
$\Var(\I_{(z_{n+1},z_n]}h) \leq Cn$. This implies the upper
bound \eqref{eqfinalBb}. The proof of Proposition \ref{propvar}
is complete.

\subsection{Proof of Proposition \ref{propdecaycorr}}

To prove Proposition \ref{propdecaycorr}, we keep the same
notations as in the previous paragraphs. The proof follows the
line of that of Theorem 2.3.6 in Gou\"{e}zel (2004c). Let $f$ be a
function in $\BV$ with $\nu(f)=0$, we wish to estimate
$\nu(|K^n f|)$ thanks to the decomposition \eqref{eq_somme}.

For the term $C_n f$, we have
\[
 \nu(|C_n(f)|) \leq C\|f\|_\infty \nu(K^n \I_{[0,z_{n+1}]})
 =C\|f\|_\infty \nu( \I_{[0,z_{n+1}]})\, .
\]
Since $\nu (J_k) \leq C/(k+1)^{1/\gamma}$, it follows that
\begin{equation}\label{1Cn}
\nu(|C_n(f)|) \leq \frac{C\|f\|_\infty}{ (n+1)^{(1- \gamma)/\gamma}} \, .
\end{equation}

We now turn to the term $\sum_{a+k+b=n}A_a E_k B_b f$ in
\eqref{eq_somme}. Let us first remark that, for any bounded
function $g$,
\[
 \nu(|A_n(g)|) \leq C\|g\|_\infty\nu(K^n \I_{(z_1,1]\cap T^{-1}[0,z_{n}]})
 =C \|g\|_\infty \nu((z_1,1]\cap T^{-1}[0,z_{n}]).
\]
Since the density of $\nu$ is bounded on $(z_1,1]$, this
quantity is $\leq C\|g\|_\infty z_{n}$. We obtain
\begin{equation}\label{1An}
\nu(|A_n(g)|) \leq \frac{C\|g\|_\infty}{ (n+1)^{1/\gamma}} \, .
\end{equation}

Using successively \eqref{1An}, \eqref{eqfinalEk} and
\eqref{eqfinalBb},  we obtain
  \begin{equation}
  \label{AEB}
  \begin{split}
  \nu\Big( \Big| \sum_{a+k+b=n} A_a E_k B_b f  \Big| \Big ) &\leq
  C \sum_{a+k+b=n}\frac {\|E_k B_b f\|_{\infty}}{(a+1)^{1/\gamma}}
  \\&
  \leq  C \sum_{a+k+b=n}\frac {\Var(f)}{(a+1)^{1/\gamma}
  (k+1)^{(1-\gamma)/\gamma}(b+1)^{1/\gamma}}
  \\&
  \leq \frac{C \Var(f)}{(n+1)^{(1-\gamma)/\gamma}}\, .
  \end{split}
  \end{equation}

We finally turn to the term
$\sum_{a+k+b=n}A_a(\I_{(z_1,1]})\cdot \nu(B_b f)$ in
\eqref{eq_somme}. From \eqref{controlcentre} and \eqref{1An},
we obtain
  \begin{equation}\label{secondtermbis}
  \begin{split}
  \nu \left(\left|\sum_{a=0}^n A_a (\I_{(z_1,1]}) \cdot \left(\sum_{b=0}^{n-a}
  \nu( B_b f)\right)\right|\right)
  &\leq
  C \Var(f) \sum_{a=0}^n \frac{1}{(a+1)^{1/\gamma}(n+1-a)^{(1-\gamma)/\gamma}}
  \\&
  \leq \frac{D \Var(f)}{(n+1)^{(1-\gamma)/\gamma}}\, .
  \end{split}
  \end{equation}

We have shown that, if $\nu(f)=0$, all the terms on the right hand side of
\eqref{eq_somme} are bounded by $C
\Var(f)/(n+1)^{(1-\gamma)/\gamma}$. Therefore, $\nu(|K^n f|)$
is bounded by the same quantity.
Now let $f$ be any $\BV$ function on $[0, 1]$, and let $\|df \|$ be the variation norm of the measure $df$ on $[0, 1]$.
To conclude the proof, it suffices to note that
$$\nu(|K^n(f^{(0)})|)\leq \frac{C
\Var(f^{(0)})}{(n+1)^{(1-\gamma)/\gamma}}\leq \frac{3C\|df\|}{(n+1)^{(1-\gamma)/\gamma}}\, . \qed $$

\section{Proofs of the main results, Theorems \ref{ASmap},
\ref{LLNmap} and \ref{LLNmap2}} \label{secglob}

It is well known that  $(T^{0},T^1, T^2, \ldots , T^{n-1})$ is
distributed as $(Y_n,Y_{n-1}, \ldots, Y_1)$ where  $(Y_i)_{i
\geq 0}$ is a stationary Markov chain with invariant measure
$\nu$ and transition kernel $K$ (see for instance Lemma XI.3 in
Hennion and Herv\'{e} (2001)). Let $X_n=f(Y_n)-\nu(f)$ for some
function $f:[0,1]\to \R$. A common argument of the proofs of
Theorems \ref{ASmap} and \ref{LLNmap} is the following
inequality: for any $\varepsilon
>0$,
\begin{equation}
\label{equ2law}
\nu \Big ( \max_{1 \leq k \leq
  n} \Big |\sum_{i=0}^{k-1} (f \circ T^i
  -\nu(f))\Big| \geq \varepsilon \Big )  \leq \nu \Big ( 2 \max_{1 \leq k \leq
  n} \Big |\sum_{i=1}^{k}X_i\Big| \geq \varepsilon \Big ) \, .
  \end{equation}
Indeed  since
\[
(f -\nu(f), f \circ T
  -\nu(f), \ldots, f \circ T^{n-1}
  -\nu(f) ) \, \text{ is distributed as } \, (X_n,X_{n-1}, \ldots, X_1),
\]
 the following equality holds in distribution
\begin{equation}
  \label{equ1law}
  \max_{1 \leq k \leq
  n} \sum_{i=0}^{k-1} (f \circ T^i
  -\nu(f))  = \max_{1 \leq k \leq
  n} \sum_{i=k}^n X_i \, .
\end{equation}  Notice now that for any $k \in [1,n]$,
\[
\sum_{i=k}^n X_i = \sum_{i=1}^n X_i - \sum_{i=1}^{k-1} X_i \, .
\]
Consequently
\[
\max_{1 \leq k \leq
  n} \Big|\sum_{i=k}^n X_i \Big| \leq  \max_{1 \leq k \leq
  n-1} \Big |\sum_{i=1}^{k}X_i\Big| + \Big |\sum_{i=1}^{n}X_i\Big| \, ,
\]
which together with \eqref{equ1law} entails \eqref{equ2law}.

\subsection{Proof of Theorem \ref{ASmap}}

According to \eqref{equ2law}, Item 1 of Theorem \ref{ASmap}
holds as soon as
\begin{equation}\label{lilbounded2}
  \sum_{n=1}^\infty \frac{1}{n} {\mathbb P} \Big( 2\max_{1 \leq k \leq
  n} \Big |\sum_{i=1}^k X_i \Big|\geq A \sqrt {n \ln(\ln (n))} \Big) < \infty \, ,
\end{equation}
for some positive constant $A$. Using the extension $(\pi_i)_{i \in {\mathbb Z}}$ of the
chain $(Y_i)_{i \geq 0}$ given at the beginning of Section \ref{pardyn},
\eqref{lilbounded2} follows from the inequality \eqref{r1app2LIL}  of Theorem \ref{ASthm}
by taking
\[
A = 40\sqrt{2} \Big (\sum_{k \geq 1} \int_0^{\alpha_{1, {\bf Y}}(k)} Q^2(u) du \Big
)^{1/2}\, .
\]
By Theorem \ref{ASthm}, \eqref{r1app2LIL}  holds as soon as $f
\in \tMonm(Q, \nu)$ and \eqref{condQ} holds. In the same way,
Item 2 of Theorem \ref{ASmap} follows from \eqref{asr1} of
Theorem \ref{ASthm} provided that \eqref{condQ} holds.

Now, by Proposition \ref{weakalpha}, $\alpha_{2, \bf
Y}(n)=O(n^{(\gamma-1)/\gamma})$. Hence \eqref{r1app2LIL} holds
as soon as, for $p=2$,
 \begin{equation}\label{DMpbis}
 f \in \tMonm (Q,\nu), \quad \text{and} \quad
\int_0^{1} u^{-\gamma (p-1)/(1- \gamma)} Q^p (u) du  < \infty \,
  .
\end{equation}
If $H$ is the c\`{a}dl\`{a}g inverse of $Q$, then $f \in \Monm(H, \nu)$
iff $f \in \tMonm(Q, \nu)$. Moreover \eqref{DMpbis} holds if
and only if
\begin{equation}\label{equDMp}
f \in \Monm (H,\nu), \quad \text{and} \quad
\int_0^{\infty} x^{p-1} (H(x))^{\frac{1-p\gamma}{1-\gamma}} dx <\infty\, .
\end{equation}
Indeed, setting $v = u^{( 1 -\gamma p)/(1- \gamma)} $, we get
that
\[
\int_0^{1} u^{-\gamma (p-1)/(1- \gamma)} Q^p (u) du  = \frac{1 -
\gamma}{1- \gamma p} \int_0^1 Q^p ( v^{(1- \gamma)/( 1 -\gamma p)} )
dv \, .
\]
Since  $H$ is  the c\`{a}dl\`{a}g  inverse of $Q$, we get
\[
\int_0^1 Q^p ( v^{(1- \gamma)/( 1 -\gamma p)} ) dv = \int_0^{\infty}
\big( H( t^{1/p}) \big )^{\frac{ 1 -p\gamma }{1- \gamma}} dt
=p\int_0^{\infty} x^{p-1} (H(x))^{\frac{1-p\gamma}{1-\gamma}} dx\, ,
\]
which concludes the proof.

\subsection{Proof of Theorem \ref{LLNmap}} By using
\eqref{equ2law}, \eqref{SLN} will hold if we can prove that for
any $\varepsilon>0$ and any $p \in (1,2)$, one has
\begin{equation}\label{SLNmarass}
  \sum_{n=1}^\infty \frac{1}{n} {\mathbb P} \Big( \max_{1 \leq k \leq
  n} \Big |\sum_{i=1}^{k}X_i\Big|\geq n^{1/p} \varepsilon \Big) < \infty \,
  .
\end{equation}
According to Theorem 4 in Dedecker and Merlev\`{e}de (2007), we
have that
\begin{equation}\label{DMp}
 \sum_{n=1}^\infty \frac{1}{n} {\mathbb P} \Big( \max_{1 \leq k \leq
  n} \Big |\sum_{i=1}^{k}X_i\Big|\geq n^{1/p} \varepsilon \Big)
  \leq C \sum_{i=0}^\infty (i+1)^{p-2} \int_0^{\gamma_i} Q_{|X_0|}^{p-1}\circ G_{|X_0|} (u) du   \, ,
\end{equation}
where  $ \gamma_i = \| {\mathbb E } (X_i   |{\mathcal M}_0 )
\|_1 $ and $G_{|X_0|}$ is the inverse of $L_{|X_0|}(x) =
\int_0^{x} Q_{|X_0|} (u) du $. We will denote by $L$ and $G$
the same functions constructed from $Q$, the c\`{a}dl\`{a}g inverse of
$H$. Assume first that $X_i = f(Y_i)-\nu(f)$ with
$f=\sum_{\ell=1}^L a_{\ell} f_{\ell}$, where $f_\ell \in
\tMon(Q, \nu)$ and $\sum_{\ell=1}^L |a_\ell| \leq 1$. According
to \eqref{majgamma}
  \begin{equation}
  \label{majgamma2} \gamma_i \leq 4 \int_0^{\alpha_{1, {\bf Y}} (i)} Q(u) du \, .
  \end{equation}
Since $Q_{|X_0|}(u)\leq Q_{|f(Y_0)|} (u) + \nu(f)$, we see that
$\int_0^x Q_{|X_0|}(u) du \leq 2 \int_0^x Q_{|f(Y_0)|}(u) du$.
Since $f=\sum a_\ell f_\ell$, we get, according to item (c) of
Lemma 2.1 in Rio (2000),
\[
\int_0^{x} Q_{|X_0|} (u) du \leq 2\sum_{\ell=1}^L \int_0^{x} Q_{|a_\ell
f_\ell(X_0)|} (u) du \leq 2\sum_{\ell=1}^L |a_\ell| \int_0^{x} Q (u) du \, .
\]
Since $\sum_{\ell=1}^L |a_\ell| \leq 1$, it follows that $
G(u/2) \leq G_{|X_0|} (u)$, where $G$ is the inverse of $x
\mapsto \int_0^{x} Q (u) du $. In particular, $G_{|X_0|}(u)\geq
G(u/4)$. Since $Q_{|X_0|}$ is non-increasing, it follows that
  \begin{align*}
  \int_0^{\gamma_i} Q_{|X_0|}^{p-1}\circ G_{|X_0|} (u) du
  &\leq \int_0^{\gamma_i} Q_{|X_0|}^{p-1} \circ G(u/4) du
  = 4 \int_0^{\gamma_i/4} Q_{|X_0|}^{p-1} \circ G(v)dv
  \\&
  =4 \int_0^{ L(\gamma_i/4)} Q_{|X_0|}^{p-1}(w) Q(w)dw
  \leq 4 \int_0^{\alpha_{1, {\bf Y}} (i)} Q_{|X_0|}^{p-1} (w) Q(w) dw   \,
  ,
  \end{align*}
where the last inequality follows from \eqref{majgamma2}. Let
$\alpha^{-1}_1 (u)=\sum_{i \geq 0} \I_{u < \alpha_{1, {\bf Y}}
(i)}$.  Since $( \alpha^{-1}_1 (u) )^{p-1} = \sum_{j \geq 0}
\big ( (j+1)^{p-1} - j^{p-1} \big ) \I_{u < \alpha_{1, {\bf
Y}}(j)}$ and $(j+1)^{p-2} \leq C\big( (j+1)^{p-1} -
j^{p-1}\bigr)$, we get
  \begin{equation}\label{DMp2bis}
  \sum_{i=0}^\infty (i+1)^{p-2}\int_0^{\gamma_i} Q_{|X_0|}^{p-1}\circ G_{|X_0|} (u) du
  \leq C \int_0^{1}(\alpha_1^{-1}(u))^{p-1} Q_{|X_0|}^{p-1} (u) Q(u) du   \,
  .
\end{equation}
Using H\"{o}lder's inequality, we derive that
\begin{multline}\label{astuceholder}
\int_0^{1}(\alpha_1^{-1}(u))^{p-1} Q_{|X_0|}^{p-1} (u) Q(u) du
\leq \Big ( \int_0^{1}(\alpha_1^{-1}(u))^{p-1}  Q^p(u) du\Big
)^{1/p}\\ \times \Big ( \int_0^{1}(\alpha_1^{-1}(u))^{p-1}
Q_{|X_0|}^{p} (u) du\Big )^{(p-1)/p} \, .
\end{multline}
Now note that $Q_{|X_0|}^{p} = Q_{|X_0|^{p} } $. By convexity
and the fact that $\sum_{\ell=1}^L |a_\ell| \leq 1$,
\begin{equation*}
Q_{|X_0|^{p} }(u)\leq Q_{\sum_{\ell=1}^L
|a_\ell||f_\ell(Y_0)-\nu(f_\ell)|^p}(u)\, .
\end{equation*}
 Using again item (c) of Lemma 2.1 in Rio (2000), we get that
\begin{equation}\label{astuceRio}
\begin{split}
 \int_0^{1}(\alpha_1^{-1}(u))^{p-1} Q_{|X_0|}^{p} (u) du &\leq
 \sum_{\ell=1}^L |a_\ell|
\int_0^{1}(\alpha_1^{-1}(u))^{p-1}
Q_{|f_\ell(Y_0)-\nu(f_\ell)|^{p}} (u) du \\
& \leq  2^{p+1} \int_0^{1}(\alpha_1^{-1}(u))^{p-1} Q^{p} (u) du \, .
\end{split}
\end{equation}
It follows that
\begin{equation}\label{DMp2ter}
\sum_{i=0}^\infty (i+1)^{p-2}\int_0^{\gamma_i} Q_{|X_0|}^{p-1}\circ G_{|X_0|} (u) du
  \leq C
   \int_0^{1}(\alpha_1^{-1}(u))^{p-1}  Q^p(u) du   \,
  .
\end{equation}

From \eqref{DMp}, \eqref{DMp2ter} and the fact that $\alpha_{1,
{\bf Y}}(n)=O(n^{(\gamma-1)/\gamma})$ by Proposition
\ref{weakalpha}, it follows that
\begin{equation*}\label{DMp3}
\sum_{n=1}^\infty \frac{1}{n} {\mathbb P} \Big( \max_{1 \leq k \leq
  n} \Big |\sum_{i=1}^{k}X_i\Big|\geq n^{1/p} \varepsilon \Big)
  \leq C
\int_0^{1} u^{-\gamma (p-1)/(1- \gamma)} Q^p (u) du  \,
  ,
\end{equation*}
and the same inequality holds for any variable $X_i=f(Y_i) -
\E(f(Y_i)) $ with
 $f \in \tMonm (Q,\nu)$ by applying Fatou's lemma. Hence \eqref{SLNmarass}
holds as soon as \eqref{DMpbis} holds. Since \eqref{DMpbis} is
equivalent to \eqref{equDMp}, the result follows.

\subsection{Proof of  Theorem \ref{LLNmap2}}
By using \eqref{equ2law}, \eqref{SLN} will hold if we can prove
that for any $\varepsilon>0$, any $p$ in $(1,2]$ and any
$b>1/p$, one has
\begin{equation}\label{SLNmarass2}
  \sum_{n=1}^\infty \frac{1}{n} {\mathbb P} \Big( \max_{1 \leq k \leq
  n} \Big |\sum_{i=1}^{k}X_i\Big|\geq n^{1/p} (\ln (n))^b \varepsilon \Big) < \infty \,
  .
\end{equation}
Let $Q$ be the c\`{a}dl\`{a}g inverse of $H$. Note that $f \in \Monm(H,
\nu)$ if and only if $f \in \tMonm(Q, \nu)$, and that $H$
satisfies \eqref{ratecond*} if and only if $Q(u) \leq
(Cu)^{-(1-p\gamma)/(p(1-\gamma))}$.

We keep the same notations as in the proof  of Theorem
\ref{LLNmap}. Assume first that $X_i = \sum_{\ell=1}^L a_{\ell}
f_{\ell}(Y_i) - \sum_{\ell=1}^L a_{\ell}\E(f_{\ell}(Y_i))$,
with $f_\ell \in  \tMonm(Q, \nu)$ and $\sum_{\ell=1}^L |a_\ell|
\leq 1$. Define the function  $(\gamma/2)^{-1}(u)=\sum_{i \geq
0} \I_{u < \gamma_i/2}$, where $ \gamma_i = \| {\mathbb E }
(X_i   |{\mathcal M}_0 ) \|_1 $. Let $\bar
R_{|X_0|}(u)=U_{|X_0|}(u) Q_{|X_0|}(u)$, with
$U_{|X_0|}=((\gamma/2)^{-1}\circ G_{|X_0|}^{-1})$. We apply
Inequality (3.9) in Dedecker and Merlev\`{e}de (2007):
\begin{align*}
  {\mathbb P} \Big( \max_{1 \leq k \leq
  n} \Big |\sum_{i=1}^{k}X_i\Big|\geq 5x \Big) &\leq
  \frac{14n}{x}\int_0^1 Q_{|X_0|}(u) \I_{x < \bar R_{|X_0|}(u)} du
   \\
  &\ \  + \frac{4n}{x^2} \int_0^1
  \I_{x \geq  \bar R_{|X_0|}(u)} \bar R_{|X_0|}(u)
  Q_{|X_0|}(u) du \, .
\end{align*}
Taking $x_n=\varepsilon n^{1/p} (\ln(n))^b/5$, and summing in
$n$, we obtain that
  \begin{align*}
  \sum_{n=1}^\infty \frac{1}{n} {\mathbb P} \Big( \max_{1 \leq k \leq n} \Big
  |\sum_{i=1}^{k}X_i\Big|\geq n^{1/p} (\ln (n))^b \varepsilon \Big)
  &\leq C \int_0^1
  \frac{\bar R^{p-1}_{|X_0|}(u)}{(\ln(\bar R_{|X_0|}(u))\vee 1)^{bp}} Q_{|X_0|}(u) du
  \\
  &\leq D  \int_0^1  \frac{U_{|X_0|}^{p-1}(u)}{(\ln(U_{|X_0|}(u))\vee 1)^{bp}}
  Q^p_{|X_0|}(u) du  \, .
\end{align*}
Now, we make the change of variables $u=G_{|X_0|}(y)$, and we
use that $G(y/2)\leq G_{|X_0|}(y)$. It follows that
\begin{equation*}
\sum_{n=1}^\infty \frac{1}{n} {\mathbb P} \Big( \max_{1 \leq k \leq n} \Big
  |\sum_{i=1}^{k}X_i\Big|\geq n^{1/p} (\ln (n))^b \varepsilon \Big) \leq
  C  \int_0^{\|X_0\|_1}  \frac{((\gamma/2)^{-1}(y))^{p-1}}{(\ln((\gamma/2)^{-1})(y)\vee 1)^{bp}}
  Q^{p-1}_{|X_0|}\circ G(y/2) dy \, .
\end{equation*}
Let $U(u)=((\gamma/2)^{-1}\circ 2G_{}^{-1})(u)$, and make the
change of variables $u=G(y/2)$. We obtain
\begin{equation*}
\sum_{n=1}^\infty \frac{1}{n} {\mathbb P} \Big( \max_{1 \leq k \leq n} \Big
  |\sum_{i=1}^{k}X_i\Big|\geq n^{1/p} (\ln (n))^b \varepsilon \Big) \leq
  C  \int_0^{1}  \frac{U^{p-1}(u)}{(\ln(U(u))\vee 1)^{bp}}
  Q^{p-1}_{|X_0|}(u) Q(u) du \, .
\end{equation*}
 From \eqref{majgamma2} we infer that $U(u)\leq C u^{-\gamma/(1-\gamma)}$, so that
 \[
\sum_{n=1}^\infty \frac{1}{n} {\mathbb P} \Big( \max_{1 \leq k \leq n} \Big
  |\sum_{i=1}^{k}X_i\Big|\geq n^{1/p} (\ln (n))^b \varepsilon \Big) \leq C
\int_0^1 \frac{u^{-\gamma(p-1)/(1-\gamma)}}{|\ln(u)|^{bp}\vee 1}Q_{|X_0|}^{p-1}(u)Q(u) du
\, .
 \]
 Applying H\"{o}lder's inequality as in \eqref{astuceholder}, and next applying
 item (c) of Lemma 2.1 in Rio (2000)
 as in \eqref{astuceRio}, it follows that
 \[
\sum_{n=1}^\infty \frac{1}{n} {\mathbb P} \Big( \max_{1 \leq k \leq n} \Big
  |\sum_{i=1}^{k}X_i\Big|\geq n^{1/p} (\ln (n))^b \varepsilon \Big) \leq C
\int_0^1 \frac{u^{-\gamma(p-1)/(1-\gamma)}}{|\ln(u)|^{bp}\vee 1}Q^p(u) du \, .
 \]
 Since $Q^p(u)
\leq (Cu)^{-(1-p\gamma)/(1-\gamma)}$,   it follows that
\begin{equation}\label{ouf!}
\sum_{n=1}^\infty \frac{1}{n} {\mathbb P} \Big( \max_{1 \leq k \leq n} \Big
  |\sum_{i=1}^{k}X_i\Big|\geq n^{1/p} (\ln (n))^b \varepsilon \Big) \leq C
\int_0^1 \frac{1}{u(|\ln(u)|^{bp}\vee 1)} du \, ,
\end{equation}
and the same inequality holds for any variable $X_i=f(Y_i) -
\E(f(Y_i)) $ with $f \in \tMonm (Q,\nu)$ by applying Fatou's
lemma. Now the right-hand term in \eqref{ouf!} is finite as
soon as $bp>1$, which concludes the proof.

\appendix
\section{Appendix}
\label{append}

We recall a maximal exponential inequality for martingales which is a straightforward
consequence of Theorem 3.4 in Pinelis (1994).

\begin{prop} \label{pinelis}
Let $( d_j, {\cal F}_j )_{j \geq 1}$ be a real-valued martingale difference sequence with
$|d_j|
 \leq c$ for all $j$. Let $M_j = \sum_{i=1}^j d_i$. Then for all $x,y
>0$,
  \begin{equation*}
  \p \Big ( \sup_{1 \leq j \leq n} |M_j| \geq x ,
  \sum_{j=1}^n \E ( \vert d_j \vert^2 | {\cal F}_{j-1} ) \leq y \Big)
  \leq
  2 \exp \left( -\frac {y}{c^2} \, h
  \Big (  \frac{xc}{y}  \Big )\right)  \, ,
  \end{equation*}
where $h(u)= (1+u) \ln (1+u) -u$.
\end{prop}

\begin{proof} Let $A_i=\{ \sum_{j=1}^i \E ( \vert d_j \vert^2 | {\cal F}_{j-1} ) \leq y \}$, and
let $\bar M_j$ be the martingale $
 \bar M_j = \sum_{i=1}^j d_i\I_{A_i}
$. Clearly
  \begin{align*}
  \p \Big ( \sup_{1 \leq j \leq n} | M_j| \geq x ,
  \sum_{j=1}^n \E ( \vert d_j \vert^2 | {\cal F}_{j-1} ) \leq y \Big)
  &=
  \p \Big ( \sup_{1 \leq j \leq n} |\bar M_j | \geq x , \sum_{j=1}^n \E ( \vert d_j \vert^2 | {\cal
  F}_{j-1} ) \leq y \Big)
  \\&
  \leq
  \p \Big ( \sup_{1 \leq j \leq n} | \bar M_j| \geq x\Big)\, .
  \end{align*}
To conclude, it suffices to apply Theorem 3.4 in Pinelis (1994) to the martingale $\bar
M_j$.
\end{proof}


\end{document}